\documentclass[11pt,a4paper]{article}
\usepackage{amsfonts}

\usepackage{graphicx,setspace}
\usepackage{amsfonts,amsmath, amssymb}
\usepackage{epstopdf}
\usepackage{color,amsmath,amssymb, amsfonts,amstext,amsthm}
\usepackage{bbm}
\usepackage{mathrsfs}
\usepackage{cite}
\usepackage{amsmath}

\newtheorem{prop}{\textbf{\ \ \quad Proposition}}[section]

\newcommand{\lam}{\lambda}
\renewcommand{\phi}{\varphi}
\newcommand{\e}{\epsilon}
\renewcommand{\a}{\alpha}

\newcommand{\mL}{\mathscr{L}}
\newcommand{\bu}{\bar{u}}

\newcommand{\R}{{\mathbb R}}
\newcommand{\de}{{\rm d}}
\newcommand{\dy}{{\rm d}y}

\newcommand{\eps}{\varepsilon}

\newcommand{\bear}{\begin{eqnarray}}
\newcommand{\enar}{\end{eqnarray}}
\newcommand{\bess}{\begin{eqnarray*}}
\newcommand{\eess}{\end{eqnarray*}}

\newcommand{\bearn}{\begin{eqnarray*}}
\newcommand{\enarn}{\end{eqnarray*}}
\newcommand{\befig}{\begin{figure}}
\newcommand{\eefig}{\end{figure}}
 
{\setlength\arraycolsep{2pt}}

\title{Numerical algorithms for mean exit time and escape probability of stochastic systems with asymmetric L\'evy motion
}
\author{Xiao Wang$^{1}$ , Jinqiao Duan$^2$, and Xiaofan Li$^2$\footnote{Corresponding author.  }, and Renming Song$^{3}$  \\
\\
\ \\
 {\small \it $^1$School of Mathematics and Statistics, Henan University}\\
  {\small \it Kaifeng 475001, China}\\
  {\small \tt email xiaoheda06@163.com}\\
   {\small \it $^2$ Department of Applied Mathematics, Illinois Institute of Technology}\\
  {\small \it Chicago, IL 60616, USA }\\
   {\small \tt email duan@iit.edu(J. Duan), lix@iit.edu(X. Li)}\\
  {\small \it $^3$ Department of Mathematics, University of Illinois}\\
  {\small \it  Urbana, IL 61801, USA}\\
    {\small \tt  email rsong@math.uiuc.edu }
}

\begin{document}
\maketitle

\begin{abstract}
For non-Gaussian stochastic dynamical systems, mean exit time
and escape probability are important deterministic quantities,
which can be obtained from integro-differential (nonlocal)
equations.
We develop an efficient and convergent numerical method
for the mean first exit time
and escape probability for stochastic systems with an asymmetric L\'evy motion,
and analyze the properties of the solutions of the nonlocal equations.
We also investigate the effects of different system factors on
the mean exit time and   escape probability,
including the skewness parameter, the size
of the domain, the drift term and the intensity of Gaussian and
non-Gaussian noises. We find that the behavior of the mean exit time and
the escape probability has dramatic difference at the boundary of the domain
when the index of stability crosses the critical value of one.
\end{abstract}

\baselineskip=15pt

\section{Introduction}

Non-Gaussian stochastic dynamical systems are found in many applications such as economics,
telecommunications and physics \cite{Middleton, Klup04, first}.
As a special non-Gaussian stochastic process, $\alpha$-stable L\'evy process
(or often called $\alpha$-stable L\'evy motion) attracts
more and more attentions of mathematicians due to the properties
which the Gaussian process does not have. For example,
the tail of a Gaussian random variable decays
exponentially which does not fit well for modeling processes with
high variability or some extreme events,
such as earthquakes or stock market crashes.
However, the stable L\'evy motion has a `heavy tail' that decays polynomially
and could be useful for these applications.
For example, financial asset returns could present heavier tails relative to
the normal distribution, and asymmetric $\a$-stable distributions are
proper alternatives for modeling  them \cite{temper1}.
 Others considered the applications
in  financial risks, physics, and biology \cite{Asy_financial_ratios, finite, Mengli}.

Recently, many researchers begin to pay attention to the stochastic
dynamical systems   with asymmetric  stable L\'evy motion
due to the demand from applications \cite{completely_asy1, Xu2013, finite}.
For instance, Lambert \cite{Lambert} considers the first passage
time and leapovers with an asymmetric  stable L\'evy motion.

In this present work, we consider the following scalar stochastic differential equation (SDE)
\begin{align} \label{SDE}
    {\rm d} X_t = f(X_t){\rm d}t  +  {\rm d} L_t,
\end{align}
where the initial condition is $ X_0 = x$, $ f $ is a drift term (vector field), and $L_t$ is a L\'evy process with the generating triplet $(0,d,\e\nu_{\a, \beta})$ (when $d$ is taken zero, it is just an asymmetric stable L\'evy motion).
Here, $\nu_{\a, \beta}$ is an asymmetric L\'evy jump measure on
$\mathbb{R}\setminus\{0\}$, to be specified in the next section.
The well-posedness of SDE driven by L\'evy motion is discussed recently.
The existence and uniqueness of solutions under the standard Lipschitz and growth conditions driven by Brownian
 motion and independent Poisson random measure were given, for example, Applebaum \cite{Apple} .
  L\"{u} {\it et al.}\cite{Lu14}
 obtained a unique solution for stochastic quasi-linear heat equation driven by anisotropic fractional L\'evy
 noises under Lipschtz and linear conditions. Chen {\it et al.}\cite{Chen15} showed
 the SDE with a large class of L\'evy process had a unique strong solution  for
 H\"{o}lder continuous drift $f$.  Priola {\it et al.}\cite{Priola12} showed the pathwise
 uniqueness for SDE driven by nondegenerate symmetric $\a$-stable L\'evy process.
We focus on the macroscopic behaviors, particularly the mean exit
time and escape probability,   for the SDE~\eqref{SDE}
with an asymmetric  stable L\'evy motion.

There are numerous works discussing the symmetric $\a$-stable L\'evy motion
and the corresponding infinitesimal generator, which is
  a nonlocal operator  and is also called the fractional Laplacian operator
$(-\triangle)^{\frac{\a}{2}}$. It is equivalent to
fractional derivative  as follows,
\[
    (-\triangle)^{\frac{\a}{2}}u(x) = \frac{{}_{-L}D_x^{\a}u(x) + {}_{x}D_L^{\a}u(x)}{2\cos{\frac{\pi \a}{2}}} \quad \text{ for } \a\neq 1,
\]
where ${}_{-L}D_x^{\a}u(x)$ and ${}_{x}D_L^{\a}u(x)$ are left and
right Riemann-Liouville fractional derivatives \cite{Huang, Yang2010}.
Various numerical methods are developed for the fractional Laplacian and
the fractional derivative operators.
To name a few,
Li {\it et al.}~\cite{SpecAppFracDer} considered the spectral
approximations to compute the fractional integral and the Caputo derivative.
Mao {\it et al.}~\cite{FPDE_Shen} developed an efficient
Spectral-Galerkin algorithms to solve fractional partial differential
equations(FPDEs).
Du {\it et al.}~\cite{Du2012} considered the general nonlocal
integral operator and provide guidance for numerical
simulations.
Qiao {\it et al.}~\cite{Asymptotic_Qiao} used  asymptotic methods to
examine escape probabilities analytically. Gao {\it et al.}~\cite{Ting12}  developed a finite difference method to compute
mean exit time and escape probability in the one-dimensional case. Our
previous work ~\cite{METXiao} proposed a method
to compute the mean exit time and escape probability for
two-dimensional stochastic systems with rotationally
symmetric $\a$-stable type L\'evy motions.

For stochastic systems with the asymmetric L\'evy motion, research on macroscopic quantities, such as mean exit time and escape probability,  is
still at its initial stage.
A couple of papers \cite{Lambert, completely_asy1}
considered the exit problem of the completely asymmetric L\'evy motion
(corresponding to $\beta=1$ or $-1$ in the jump measure $\nu_{\a,\beta}$).
A few of people have studied the processes for their basic properties.
Considering a completely asymmetric L\'evy process that has
absolutely continuous transition probabilities,
 Bertoin ~\cite{exponential_decay_asy} proved the decay
and ergodic properties of the transition probabilities while
Lambert~\cite{Lambert} established the existence of
the L\'evy process conditioned to stay in a finite interval.
Koren {\it et al.}~\cite{first}
investigated the first passage times and the first
passage leapovers of symmetric and completely asymmetric
L\'evy stable random motions.

The paper is organized as follows. In Section 2, we review the concepts of asymmetric
L\'evy motion, mean exit time and escape probability, and show the symmetry of solutions
to the exit problem. A numerical method and simulation results for mean exit time and escape probability are
presented in Section 3 and 4, respectively. Finally, Section 5 presents the conclusion of our paper.

\section{Concepts}
\label{sec:1}
\subsection{Asymmetric L\'evy motion}

Stable distribution, denoted by $S_{\a}(\sigma, \beta, \mu)$,
is a four-parameter family of distributions with
$\a \in (0, 2], \sigma \geq 0, \beta \in [-1, 1]$ and  $\mu \in \R$.
Usually $\a$ is
called the index of stability (or non-Gaussianity index), $\sigma$
is the scale parameter, $\beta$ is the skewness parameter and $\mu$
is the shift parameter.
It is said to be completely asymmetric
if $\beta = \pm 1$ \cite{Taqqu, Sato-99, Apple}.
Stable distribution and its profile help us understand
the behavior of the process governed by the SDE \eqref{SDE},
because $L_1$ is a random variable with the probability density functions(PDFs)
of $S_{\a}(1, \beta, 0)$. The corresponding generating triplet
is $(K_{\a,\beta},0,\nu_{\a, \beta})$,   where the constant $K_{\a,\beta}$ and the jump measure
$\nu_{\a, \beta}$ are defined below in \eqref{K}
and \eqref{asy_measure} respectively.
Some examples of $S_{\a}(1, \beta, 0)$ are shown in Figure~\ref{skewpdf}.
\begin{figure}[h]
  \includegraphics[width=\linewidth]{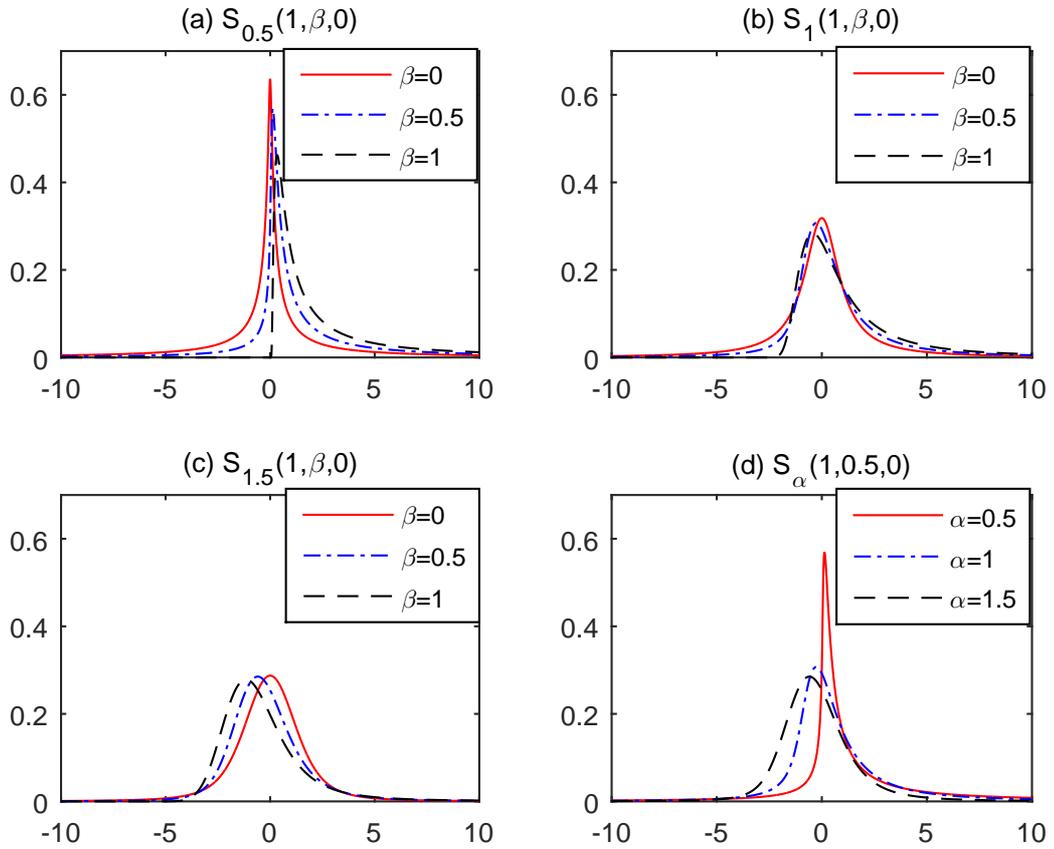}
\caption{Probability density functions $S_{\a}(1,\beta,0)$ of $L_1$
for different values of $\a$ and $\beta$.
}
\label{skewpdf}       
\end{figure}

For the $\a$-stable L\'evy motion, we have the corresponding  L\'evy-Khinchin formula \cite{Sato-99, Tankov, asymmetric_Hein}
\bear \label{Charfun}
    \mathbb{E}(e^{i\lam L_t}) =
    \begin{cases}
        \mbox{exp}\{-\sigma^{\a}|\lam|^{\a}t(1-i\beta \mbox{sgn}{\lam} \tan{\frac{\pi \a}2}) + i\mu \lam t \},\; \mbox{for} \; \a \neq 1,\\
        \mbox{exp}\{-\sigma|\lam|t(1+i\beta \frac{2}{\pi} \mbox{sgn}{\lam} \log{|\lam|}) + i\mu \lam t \}, \quad \mbox{for} \;\a = 1 .
    \end{cases}
\enar

For every $\varphi \in H_0^{2}(\R)$, we can obtain the generator for the solution
to the SDE \eqref{SDE} with the asymmetric
$\a$-stable L\'evy motion $L_t$  from the above formula (\ref{Charfun}) as \cite{Sato-99, Duan}
\begin{align}\label{generator}
     \mL \varphi (x) = & (f(x)+ \e K_{\a,\beta}) \varphi'(x)
    + \frac{d}{2}\varphi''(x)\notag \\
      &+ \e \int_{\mathbb{R} \setminus\{0\}} \left(\varphi(x+ y)-\varphi(x)
     -  1_{B} \; y \varphi'(x)\right) \nu_{\a,\beta}({\rm d}y),
\end{align}
where the measure $\nu_{\a,\beta}$ is given by
\bear  \label{asy_measure}
   \nu_{\a,\beta}({\rm d}y) = \frac{C_1 1_{\{0<y<\infty\}}(y)+C_2 1_{\{-\infty<y<0\}}(y)}{|y|^{1+\alpha}}({\rm d}y)
\enar
with
\bear \label{C1C2}
     C_1 = C_{\a}\frac{1+\beta}2, C_2 = C_{\a}\frac{1-\beta}2,\quad -1 \leq \beta \leq 1,
\enar
and
\bear  \label{C_alpha}
    C_{\a} =
    \begin{cases}
        \frac{\a(1-\a)}{\Gamma(2-\a)\cos{(\frac{\pi \a}2)}}\;,   &\text{ $ \a \neq 1 $;}\\
        \frac2{\pi},  \;  &\text{ $ \a = 1$.}
    \end{cases}
\enar
In Eq.~\eqref{generator} the constants $d$ and $\e$ represent the intensities of Gaussian and L\'evy noises respectively.

The constant $K_{\a,\beta}$ in \eqref{generator} is given by
\bear  \label{K}
    K_{\a,\beta}=
    \begin{cases}
        \frac{C_1-C_2}{1-\a} \;,   &\text{ $ \a \neq 1 $;}\\
        (\int_1^{\infty}\frac{\sin(x)}{x^2}{\rm d}x + \int_0^{1}\frac{\sin(x)-x}{x^2}{\rm d}x)(C_2-C_1),  \;  &\text{ $ \a = 1$.}
    \end{cases}
\enar

Furthermore, we note that $\beta = \frac{C_1-C_2}{C_1+C_2} $,
and the symmetries $C_1(-\beta)= C_2(\beta)$
and $K_{\a,-\beta} = -K_{\a,\beta}$. We point out that
when $\a \neq 1$ the stable distribution is strictly $\a$-stable, while
it is strictly $\a$-stable when $\a=1$
if and only if its L\'evy measure is symmetric.

We remark that, for people who are not too familiar with stable distributions,
the non-solid curves in
Fig.~\ref{skewpdf}(b) and (c) might be counter intuitive.
After all, in both figures,
the skewness parameter $\beta$ is positive in all these case and thus there is a
bigger tendency of jumping to the right, and yet these curves are shifted to left
near the origin. This is due to the compensation which produces a linear drift
with coefficient $K_{\a,\beta}$ given by \eqref{K}.
In all these cases, $K_{\a,\beta}$ is negative.

\subsection{Mean exit time and escape probability}
 The exit time problem is important in many fields, such as physiscs, finance and economics. The first exit time starting at $x$ from a bounded domain $D$ is defined as $ \tau{(\omega)}:= \inf \{t \geq 0, X_{t}(\omega , x) \notin  D \} $, and the mean first exit time (MET) is $u(x)=\mathbb{E}[\tau(\omega)].$


Assume that $f(x)$ satisfies Lipschitz condition and linear growth condition for the existence and uniqueness of solution \cite{Apple}. Due to the Dynkin's formula, the MET $u$ satisfies the following nonlocal partial differential equation \cite{Duan}
\begin{align}
    \mL u(x) &= -1, \quad \text{for} \; x \in D,
\label{exit}
\end{align}
  subject to the Dirichlet-type exterior condition,
\begin{align}\label{dec}
     u(x) &= 0 ,\quad  \text{for} \; x \in D^c,
\end{align}
where $\mL$ is the generator defined in (\ref{generator}) and $D$ is open.

 Consider the escape probability of the process $X_t $ in the SDE (\ref{SDE}). The escape probability from $D$ to $E$ is the likelihood that $X_t$ with the initial location $X_0 = x$, exits from $D$ and first lands in $E$ which belong to $D^c$, denoted as $P_E(x)=\mathbb{P}\{X_{\tau}\in E\}$. The escape probability satisfies the following nonlocal partial differential equation \cite{Huijie, Duan}
\begin{gather}
    \mathscr{L}\, P_E(x)=0, \quad x \in D, \label{eq.ep}  \nonumber\\
    P_E|_{x \in E}=1, \quad  P_E|_{x \in D^c\setminus E}=0. \label{eq.epec}
\end{gather}

\subsection{Symmetry and non-dimensionalization}
For the domain $D = (-b,b)$, the MET $u$ satisfies Eq.~(\ref{exit}).
We replace the $I_{\{|y|< 1 \}}(y)$ in Eq.~(\ref{exit}) to $I_{\{|y|< b \}}(y)$ and get \cite{Sato-99}
\begin{eqnarray}  \label{asymmetricEq2}
  \frac{d}{2} u''(x) + c(x) u'(x)
 + \eps \int_{\R \setminus\{0\}} [u(x+y)-u(x)
     -  I_{\{|y|< b \}}(y) \; y u'(x)]  \\
     \left[\frac{C_1 1_{\{0<y<\infty\}}+C_2 1_{\{-\infty<y<0\}}}{|y|^{1+\alpha}}\right]\; {\rm d}y  = -1,
\nonumber
\end{eqnarray}
where
\bear
    c(x) =
    \begin{cases}
        f(x) + \e K_{\a,\beta} + \eps(C_1-C_2)\frac{b^{1-\a}-1}{1-\a}\;,   &\text{ $ \a \neq 1 $;}\\
        f(x) + \e K_{\a,\beta} + \eps(C_1-C_2)\ln{b},  \;  &\text{ $ \a = 1$.}
    \end{cases}
\enar \label{cs}

Next, we show the solution to the MET problem has the following symmetry
when $f$ is an odd function. However, the
numerical method presented in the work does not require
$f$ be odd. Because of the application in dynamical systems, we focus on
the O-U potential($f(x) = -x$) later.

\begin{prop} [Symmetry of Solutions] \label{thm}
If $f(x)$ is an odd function and the domain $D$ is symmetric
about the origin ($D=(-b,b)$), then the MET $u$
(or, equivalently, the solution $u$ to Eq.~(\ref{asymmetricEq2}))
is symmetric about the origin
if $\beta$ changes the sign, i.e. $u_{-\beta}(-x) = u_\beta(x)$ for all $x\in (-b, b)$ where $u_\beta$ and $u_{-
\beta}$ denote the solutions corresponding to $\beta$ and $-\beta$ respectively.
\end{prop}

\begin{proof}
 Since Eq.~(\ref{asymmetricEq2}) is valid for all $-1\leqslant \beta \leqslant 1$
 and $-b\leqslant x \leqslant b$, $u_{-\beta}(-x)$ satisfies the following equation,

\bearn
    && \frac{d}{2} u_{-\beta}''(-x) + c(-x) u_{-\beta}'(-x)   \\
 +&& \eps \int_{\R \setminus\{0\}} [u_{-\beta}(-x+y)-u_{-\beta}(-x)
     -  I_{\{|y|< b \}}(y) \; y u_{-\beta}'(-x)] \nu_{\a, -\beta}( {\rm d}y) =-1. \\
\enarn
Define $\bar{u}(x) = u_{-\beta}(-x)$. We can see $\bar{u}'(x) = -u_{-\beta}'(-x)$. Taking $y' = -y$, we have
\bearn
    &&\int_{\R \setminus\{0\}} [u_{-\beta}(-x+y)-u_{-\beta}(-x) -  I_{\{|y|< b \}}(y) \; y u_{-\beta}'(-x)] \nu_{\a, -\beta}( {\rm d}y) \\
     &=& \int_{\R \setminus\{0\}} [\bu(x-y)-\bu(x) +  I_{\{|y|< b \}}(y) \; y \bu'(x)]
        \left[\frac{C_2 1_{\{0<y<\infty\}}+C_1 1_{\{-\infty<y<0\}}}{|y|^{1+\alpha}}\right]\; {\rm d}y  \\
     &=& \int_{\R \setminus\{0\}} [\bu(x+y')-\bu(x) -  I_{\{|y'|< b \}}(y') \; y' \bu'(x)]
        \left[\frac{C_1 1_{\{0<y'<\infty\}}+C_2 1_{\{-\infty<y'<0\}}}{|y'|^{1+\alpha}}\right]\; {\rm d}y'.
\enarn

When $f(-x) = -f(x)$, we have
\bearn
    c(-x) =
    \begin{cases}
        -f(x) + \e K_{\a,\beta} + \eps C_{\a}\beta\frac{b^{1-\a}-1}{1-\a}\;,   &\text{ $ \a \neq 1 $;}\\
        -f(x) + \e K_{\a,\beta} + \eps C_{\a}\beta \ln{b},  \;  &\text{ $ \a = 1$.}
    \end{cases}
\label{cs1}
\enarn
Thus $c_{-\beta}(-x) = - c_\beta(x)$ if $f$ is an odd function,
where $c_\beta$ and $c_{-\beta}$ denote the
function $c$ corresponding to $\beta$ and $-\beta$ respectively.

Using
\[
    \frac{d}{2} u_{-\beta}''(-x) + c_{-\beta}(-x) u_{-\beta}'(-x) = \frac{d}{2} \bu''(x) + c_\beta(x) \bu'(x),
\]
we get,
\bearn
    \frac{d}{2} \bu''(x) + c(x) \bu'(x) + \eps \int_{\R \setminus\{0\}} [\bu(x+y')-\bu(x)
     -  I_{\{|y'|< b \}}(y') \; y' \bu'(x)]  \\
     \left[\frac{C_1 1_{\{0<y'<\infty\}}+C_2 1_{\{-\infty<y'<0\}}}{|y'|^{1+\alpha}}\right]\; {\rm d}y'  = -1.
\enarn
Thus, we have shown $\bar{u}(x)$ satisfies the same Eq.~(\ref{asymmetricEq2}) if $f(-x) = -f(x)$.
Due to uniqueness of the solution, we have $u_{-\beta}(-x) = u_\beta(x)$.     $ \Box $
\end{proof}

To keep the computational domain fixed as $[-1,1]$,
we perform the change of variable
\begin{equation}
s = x/b, \quad \text{ and } v(s) := u(bs).
\label{eq.cov}
\end{equation}
Then, $\displaystyle{
    \frac{ {\rm d} u}{ {\rm d} x} = \frac1b \frac{{\rm d} v}{{\rm d} s}, \;  \frac{\de^2 u}{\de x^2} = \frac1{b^2} \frac{\de^2 v}{\de s^2}
}$.
Let $y = br$, we have
\bess
     && \int_{\R \setminus\{0\}} [u(x+y)-u(x)
     -  I_{\{|y|< b \}}(y) \; y u'(x)]
     \left[\frac{C_1 1_{\{0<y<\infty\}}+C_2 1_{\{-\infty<y<0\}}}{|y|^{1+\alpha}}\right]\; {\rm d}y \\
     &=&b^{-\a} \int_{\R \setminus\{0\}} [v(s+r)-v(s)
     -  I_{\{|r|< 1 \}}(r) \; r  v'(s)]
    \left[\frac{C_1 1_{\{0<r<\infty\}}+C_2 1_{\{-\infty<r<0\}}}{|r|^{1+\alpha}}\right] {\rm d}r
\eess
Finally, the equation for the MET (\ref{asymmetricEq2}) becomes
\begin{eqnarray}  \label{asymmetricEq3}
  &&\frac{d}{2b^2} \frac{\de^2 v}{\de s^2} + \frac{c(bs)}{b} \frac{\de v}{\de s} \nonumber\\
 &+&\eps b^{-\a} \int_{\R \setminus\{0\}} [v(s+r)-v(s)
     -  I_{\{|r|< 1 \}}(r) \; r v'(s)]
    \left[\frac{C_1 1_{\{0<r<\infty\}}+C_2 1_{\{-\infty<r<0\}}}{|r|^{1+\alpha}}\right] {\rm d}r  \nonumber\\
 &=& -1.
\end{eqnarray}

\section{Numerical methods}

In this section, we describe the numerical methods for solving the MET $v(s)$
in Eq.~\eqref{asymmetricEq3} on the fixed computational domain $s\in (-1,1)$.
The solution for the MET $u$ in the original equations (\ref{exit})
and \eqref{dec}
for the symmetric domain $D=(-b,b)$ is obtained from $u(x) \equiv v(x/b)$.

\subsection{Reformulation}

Before we present our numerical schemes, we first reformulate
the integral in \eqref{asymmetricEq3}, denoted by
\begin{eqnarray}  \label{asymmetricEq}
  I := \int_{\R \setminus\{0\}} [v(s+r)-v(s)
     -  I_{\{|r|< 1 \}}(r) \; r v'(s)]  
     \frac{C_1 1_{\{0<r<\infty\}}(r)+C_2 1_{\{-\infty<r<0\}}(r)}{|r|^{1+\alpha}}\; {\rm d}r.
\end{eqnarray}
We decompose $I = C_1 I_1 + C_2 I_2$, where
\bear \label{DecomI1}
     I_1&=&\int_{\R^+}\frac{v(s+r)-v(s)
     -  I_{\{|r|< 1 \}}(r) \; r v'(s)}{|r|^{1+\alpha}}\; {\rm d}r , \\
      I_2 &=& \int_{\R^-}\frac{v(s+r)-v(s)
     -  I_{\{|r|< 1 \}}(r) \; r v'(s)}{|r|^{1+\alpha}}\; {\rm d}r .
\enar

Using the condition \eqref{dec} exterior to the domain $D$, i.e., $v(s)$ vanishes when $|s| \geq 1$, we obtain
\bear
    I_1 = -\frac{v(s)}{\a}(1-s)^{-\a} - v'(s)g(s)
          + \int_0^{1-s} \frac{v(s+r)-v(s)-r v'(s)}{r^{1+\a}} {\rm d}r,
\label{i1p}
\enar
for $s > 0$;
\begin{eqnarray}
    I_1 &=& -\frac{v(s)}{\a}(1-s)^{-\a} + \int_1^{1-s} \frac{v(s+r)-v(s)}{r^{1+\a}} {\rm d}r \nonumber \\
         &&+ \int_0^1 \frac{v(s+r)-v(s)-rv'(s)}{r^{1+\a}} {\rm d}r,
\label{DecomI1s}
\end{eqnarray}
for $s<0$,
where
\bear
    g(s) =
    \begin{cases}
        \frac{1-(1-|s|)^{1-\a}}{1-\a}\;,   &\text{ $ \a \neq 1 $;}\\
         -\ln(1-|s|),  \;  &\text{ $ \a = 1$.}
    \end{cases}
\label{g1}
\enar
Similarly,
\bear
  I_2 &=& \int_1^{1+s} \frac{v(s-y)-v(s)}{y^{1+\a}} \; {\rm d}y - \frac{v(s)}{\a} (1+s)^{-\a} \nonumber\\
   &&+\int_0^1 \frac{v(s-y)-v(s)+y v'(s)}{y^{1+\a}} \;{\rm d}y,
\label{i2p}
\enar
for $s>0$;
\bear
   I_2 = -\frac{v(s)}{\a} (1+s)^{-\a} + v'(s)g(s)
   + \int_0^{1+s} \frac{v(s-y)-v(s)+y v'(s)}{y^{1+\a}} \; {\rm d}y,
\label{i2n}
\enar
for $s<0$.

Now, combining the above results \eqref{i1p}--\eqref{i2n},
we rewrite (\ref{asymmetricEq3}) as following
\bear \label{LastEq1}
    \frac{d}{2b^2} v''(s) &+& \left(\frac{c(bs)}{b}- \eps b^{-\a} C_1g(s) \right) v'(s)
     - \eps b^{-\a}\frac{v(s)}{\a}\left[ C_1(1-s)^{-\a} + C_2(1+s)^{-\a}\right]  \nonumber\\
     &+& \eps b^{-\a} C_1\int_0^{1-s} \frac{v(s+r)-v(s)-r v'(s)}{r^{1+\a}} {\rm d}r +
     \eps b^{-\a} C_2\int_1^{1+s} \frac{v(s-y)-v(s)}{y^{1+\a}} \; {\rm d}y \nonumber\\
      &+& \eps b^{-\a} C_2 \int_0^1 \frac{v(s-y)-v(s)+y v'(s)}{y^{1+\a}} \;{\rm d}y = -1,
\enar
for $s \geq 0$, while
\bear \label{LastEq2}
    \frac{d}{2b^2} v''(s) &+& \left(\frac{c(bs)}{b}+ \eps b^{-\a} C_2g(s) \right) v'(s)
     -\eps b^{-\a}\frac{v(s)}{\a}\left[ C_1(1-s)^{-\a} +  C_2(1+s)^{-\a}\right]  \nonumber\\
     &+& \eps b^{-\a} C_1\int_0^{1} \frac{v(s+r)-v(s)-r v'(s)}{r^{1+\a}} {\rm d}r +
      \eps b^{-\a} C_1 \int_1^{1-s} \frac{v(s+r)-v(s)}{r^{1+\a}} \; {\rm d}r \nonumber\\
      &+& \eps  b^{-\a} C_2\int_0^{1+s} \frac{v(s-y)-v(s)+y v'(s)}{y^{1+\a}} \;{\rm d}y = -1,
\enar
for $s<0$.

For completeness, we provide the equations for finding the escape probability
in Eq.~(\ref{eq.epec}). They are different than those for the MET
due to the difference in the exterior condition. To be specific, we take
$D = (-b,b)$ and $E = [b,\infty)$, then the Eq.~(\ref{eq.epec})
becomes
\bear \label{LastEq11}
    \frac{d}{2b^2} v''(s) &+& \left(\frac{c(bs)}{b}- \eps b^{-\a} C_1g(s) \right) v'(s)
     - \eps b^{-\a}\frac{v(s)}{\a}\left[ C_1(1-s)^{-\a} + C_2(1+s)^{-\a}\right]  \nonumber\\
     &+& \eps b^{-\a} C_1\int_0^{1-s} \frac{v(s+t)-v(s)-t v'(s)}{t^{1+\a}} {\rm d}t +
     \eps b^{-\a} C_2\int_1^{1+s} \frac{v(s-y)-v(s)}{y^{1+\a}} \; {\rm d}y \nonumber\\
      &+& \eps b^{-\a} C_2 \int_0^1 \frac{v(s-y)-v(s)+y v'(s)}{y^{1+\a}} \;{\rm d}y = -\frac{\eps b^{-\a} C_1}{\a}(1-s)^{-\a},
\enar
for $s \geq 0$, and
\bear \label{LastEq12}
    \frac{d}{2b^2} v''(s) &+& \left(\frac{c(bs)}{b}+ \eps b^{-\a} C_2g(s) \right) v'(s)
     -\eps b^{-\a}\frac{v(s)}{\a}\left[ C_1(1-s)^{-\a} +  C_2(1+s)^{-\a}\right]  \nonumber\\
     &+& \eps b^{-\a} C_1\int_0^{1} \frac{v(s+t)-v(s)-t v'(s)}{t^{1+\a}} {\rm d}t +
      \eps b^{-\a} C_1 \int_1^{1-s} \frac{v(s+t)-v(s)}{t^{1+\a}} \; {\rm d}t \nonumber\\
      &+& \eps  b^{-\a} C_2\int_0^{1+s} \frac{v(s-y)-v(s)+y v'(s)}{y^{1+\a}} \;{\rm d}y = -\frac{\eps b^{-\a} C_1}{\a}(1-s)^{-\a},
\enar
for $s<0$.

\subsection{Discretization}
We are ready to describe our discretization based on the formulation
in the equations \eqref{LastEq1}
and \eqref{LastEq2}.
We divide the computational domain $[-1,1]$ by $2J$ subintervals:
 $s_j = jh, -J \leq j \leq J$ with each subinterval having
the size $h = 1/J$.
Denote the numerical solution to the unknown MET $v$
by the vector $\mathbf{V} = V_{-J:J}$,  where the component $V_j$
approximates $v_j \equiv v(s_j)$ for $-J \leq j \leq J$. Note that
$V_{-J}=V_{J}=0$ from the exterior condition \eqref{dec}.

The singular integrals in Eqs.~(\ref{LastEq1}) and \eqref{LastEq2}
need special quadrature rules.
Following the quadrature error analysis of Sidi and Israeli \cite{Sidi}, we have
the following leading-order error for the "punch-hole" trapezoidal rule
for the weakly singular integrals in (\ref{LastEq1})
\bear
    &&C_1\int_0^{1-s} \frac{v(s+r)-v(s)-r v'(s)}{r^{1+\a}} {\rm d}r +
    C_2\int_0^1 \frac{v(s-r)-v(s)+r v'(s)}{r^{1+\a}} {\rm d}r \nonumber \\
    &=& h \sum\limits_{j=1}^{J_{1-s}}\!{'} G_1(r_j) +  h \sum\limits_{k=1}^{J}\!{'}G_2(r_k) +C_p.
\label{eq.site}
\enar
where $G_1(r) = \dfrac{v(s+r)-v(s)-r v'(s)}{r^{1+\a}}$,
$G_2(r) = \dfrac{v(s-r)-v(s)+r v'(s)}{r^{1+\a}}$ and the leading-order errors
are
\bear
    C_p &=& -C_\a \zeta(\a-1)\frac{v''(s)}2 h^{2-\a} - \beta C_\a \zeta(\a-2)\frac{v^{'''}(s)}6 h^{3-\a} \nonumber\\
    &&-[C_1 G_1'(1-s) +C_2 G_2'(1)]\frac{B_2}{2}h^2 + \mathcal {O}(h^{4-\a}).
\label{eq.tee}
\enar
We denote $\displaystyle{\sum\limits_{k=1}^{J}\!{'}}$ as the summation
where the term with the upper limit $k=J$ is multiplied by $1/2$
and $J_{1-s}$ is the index corresponding to $1-s$.
$\zeta$ is the Riemann zeta function.

Define
\begin{equation}
C_h = \frac{d}{2b^2}-\frac{\eps b^{-\a}C_{\a}\zeta(\a-1)}{2}h^{2-\a}.
\label{eq.ch}
\end{equation}
Using central differencing for the derivatives
and modifying the "punched-hole" trapezoidal rule with the leading
order term $O(h^{2-\a})$ in \eqref{eq.tee}, we get
the $j$-th equation discretizing the right-hand side(RHS) of (\ref{LastEq1})
\begin{align}\label{eq:Discrete1}
    L_{j, :}\mathbf{V}:=& C_h\frac{V_{j+1}-2V_j+V_{j-1}}{h^2}+\left(\frac{c(bs_j)}{b}- \eps b^{-\a} C_1g(s_j)\right) \frac{V_{j+1}-V_{j-1}}{2h} \nonumber\\ 
     &- \eps b^{-\a} \frac{V_j}{\a}\left[ C_1(1-s_j)^{-\a}+C_2(1+s_j)^{-\a}\right] \nonumber  \\
    &+ \eps b^{-\a} C_1 h\sum_{k=j+1}^{J}\!\!{'} \frac{V_k-V_j-(s_k-s_j)\frac{V_{j+1}-V_{j-1}}{2h}}{(s_k-s_j)^{\a+1}}
    + \eps b^{-\a} C_2 h\sum_{k=-J}^{-J+j}\;\;\!\!\!\! \frac{V_k-V_j}{(s_j-s_k)^{1+\a}} \nonumber\\
    &+ \eps b^{-\a} C_2 h\sum_{k=-J+j}^{j-1}\!\!\!\!{''} \frac{V_k-V_j-(s_k-s_j)\frac{V_{j+1}-V_{j-1}}{2h}}{(s_j-s_k)^{\a+1}}
\end{align}
for $0 \leq j\leq J-1$
where the summation symbol $\sum$
means the terms of both end indices are multiplied by $\frac12$,
$\sum{''}$ means that only the term of the bottom index
is multiplied by $\frac12$.
Similarly,
\begin{align}\label{eq:Discrete2}
    L_{j, :}\mathbf{V}:=&C_h \frac{V_{j+1}-2V_j+V_{j-1}}{h^2}  
    + \left(\frac{c(bs_j)}{b}+ \eps b^{-\a} C_2g(s_j)\right) \frac{V_{j+1}-V_{j-1}}{2h} \nonumber \\
     &- \eps b^{-\a} \frac{V_j}{\a}\left[ C_1(1-s_j)^{-\a}+C_2(1+s_j)^{-\a}\right] \nonumber  \\
    &+ \eps b^{-\a} C_1 h\sum_{k=j+1}^{J+j}\!\!\!{'} \frac{V_k-V_j-(s_k-s_j)\frac{V_{j+1}-V_{j-1}}{2h}}{(s_k-s_j)^{\a+1}}
    + \eps b^{-\a} C_1 h\sum_{k=J+j}^{J}\;\;\!\!\!\! \frac{V_k-V_j}{(s_k-s_j)^{1+\a}} \nonumber\\
    &+ \eps b^{-\a} C_2 h\sum_{k=-J}^{j-1}\!\!\!{''} \frac{V_k-V_j-(s_k-s_j)\frac{V_{j+1}-V_{j-1}}{2h}}{(s_j-s_k)^{\a+1}}
\end{align}
for $-J+1 \leq j \leq -1$.

We can write the discretized equations (\ref{eq:Discrete1}) and (\ref{eq:Discrete2})
simply as $L\; V_{-J+1:J-1} = - \vec{1}$, where $L$ is the $(2J-1)$ by $(2J-1)$
coefficient matrix, $\vec{1}$ is the vector of ones with dimension $2J-1$.
The dense system of linear equations is solved by the Krylov-subspace iterative
method GMRES\cite{GMRES}. We point out that, for $0 < \a \leq 1$,
we use one-sided finite difference formula for derivatives
of $u$ at the grid points that are closest to the end points $s=\pm 1$,
 because our results show that the solutions could be discontinuous at
the end points in this case.

\section{Numerical results}

\subsection{Validation}

\befig[h]
\includegraphics*[width = \linewidth]{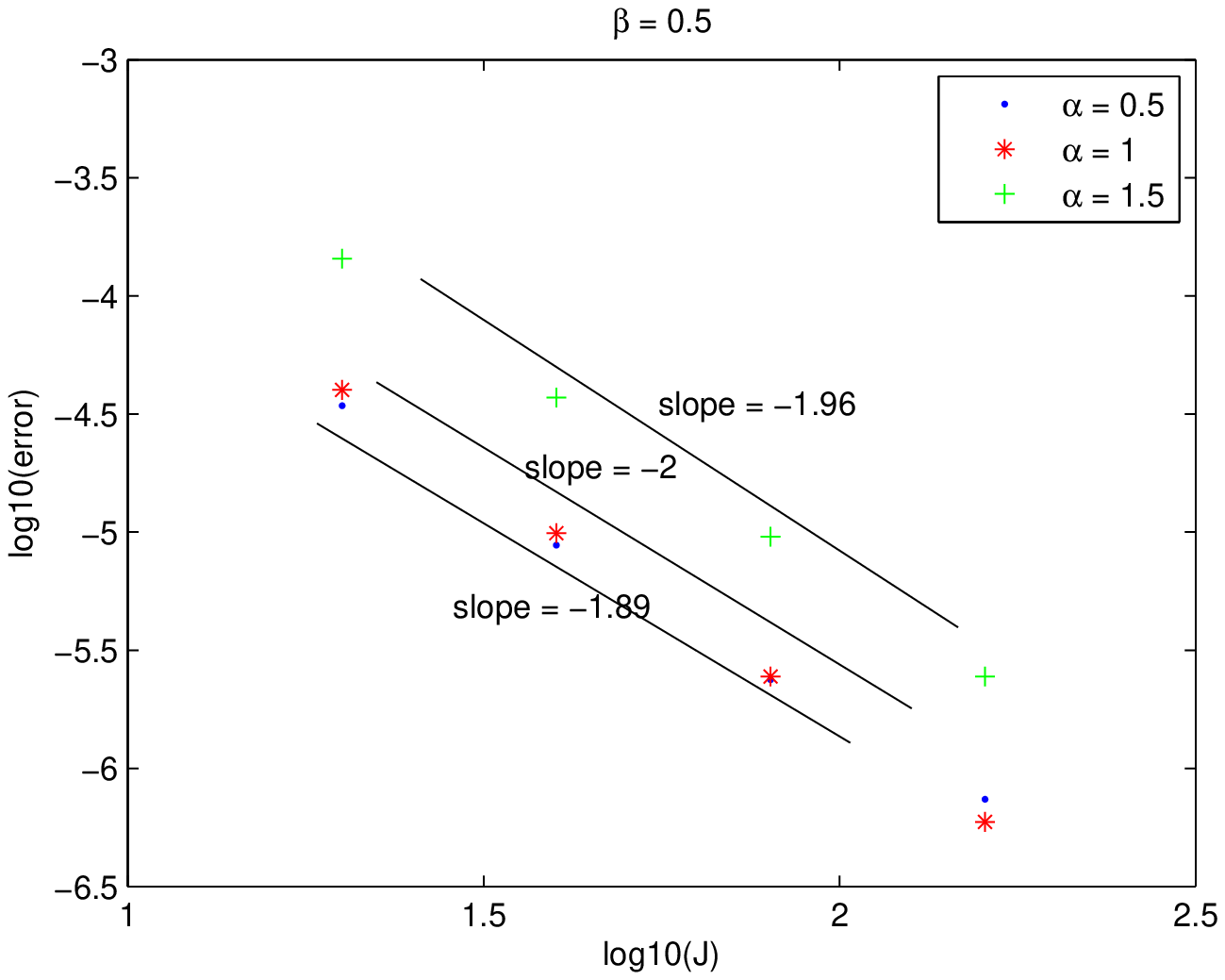}
\caption{The error of the numerical solutions to the constructed equation
$\mL u = \mL (1+x^2)_+$ with the RHS given in Eqs.~\eqref{rhs} (for $\a\neq 1$)
or \eqref{rhs1} (for $\a=1$),   as a function of the resolution $J$.
The errors evaluated at $x=-0.5$ are shown for the different values
of $\a=0.5, 1, 1.5$
but the fixed value of $\beta=0.5, d = 0, f \equiv 0, \eps = 1$.}
\label{CheckOrder}
\eefig
Since we are not aware of any explicit exact solution for $\beta \neq 0$,
we let $u(x) = (1-x^2)_+$, i.e. $u(x) = 1-x^2 $ for $x \in (-1, 1)$ and
$u(x) = 0 $ otherwise, together with
$d = 0, f \equiv 0, \eps = 1$ and the domain $D = (-1,1)$.
We compute $\mL u$ where the operator $\mL$ is defined in \eqref{generator}
or equivalently the left-hand side(LHS) of Eq.~(\ref{asymmetricEq2})
and obtain
\bear \label{rhs}
   \mL (1-x^2)_+ &=&
        C_1\left[-\frac{(1-x)^{2-\a}}{2-\a}-\frac{2x((1-x)^{1-\a}-1)}{1-\a}
          -\frac{(1+x)(1-x)^{1-\a}}{\a}\right]  \\
     &&   +C_2 \left[-\frac{(1+x)^{2-\a}}{2-\a}-\frac{2x(1-(1+x)^{1-\a})}{1-\a}
       -\frac{(1-x)(1+x)^{1-\a}}{\a}\right],  \nonumber
\enar
for $ \a \neq 1 $;
\bear \label{rhs1}
   \mL (1-x^2)_+ & = &
         -2(C_1+C_2)-2x[C_1\ln(1-x)-C_2\ln(1+x)],
\enar
for $\a=1$.

Replacing the RHS $-\mathbf{1}$ of the MET equation \eqref{exit} by $\mL (1+x^2)_+$ given in \eqref{rhs} and \eqref{rhs1},
we have created a known solution $u(x) = (1-x^2)_+$.
We compare the numerical solutions using
our discretizations \eqref{eq:Discrete1}
and \eqref{eq:Discrete2} (with the new RHS)
against the analytical expression
$u(x) = (1-x^2)_+$ for different resolutions $J = 20, 40, 80, 160$.
Figure~\ref{CheckOrder} shows that the numerical order of convergence
based on the computed errors is close to
two for all three values of $\a=0.5, 1, 1.5$ tested
at the fixed point $x= -0.5$. The convergence order is two, expected from the
error analysis of our numerical method. In the verification, we have chosen
$\beta=0.5, d=0, f\equiv 0, \eps =1$.

\befig[h]
\includegraphics[width = \linewidth]{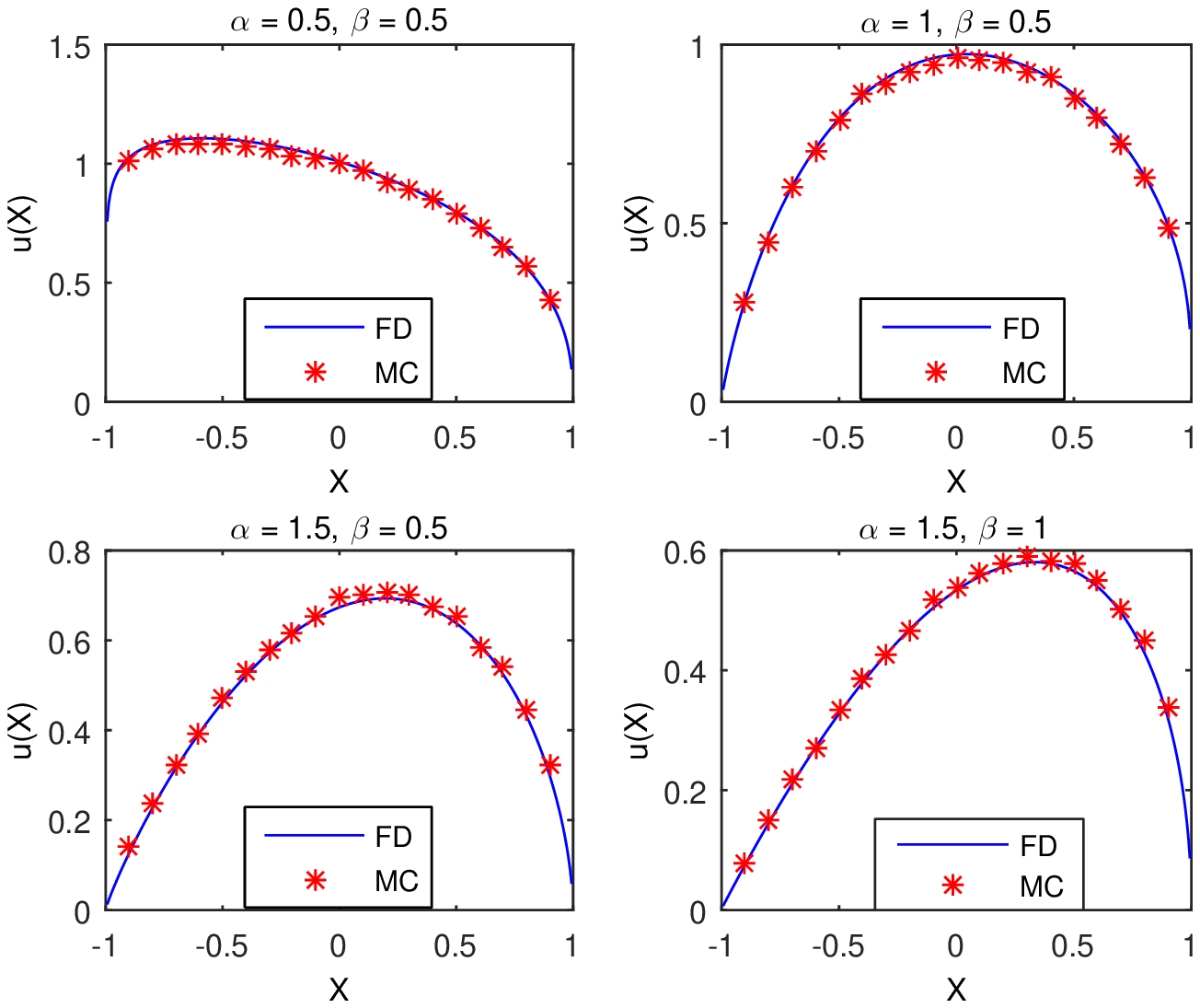}
\caption{Comparison between the two numerical solutions
of the MET for $f\equiv 0, d=0, \eps=1$ and domain $D = (-1,1)$,
one from our numerical method (labeled 'FD') and the other from
the direct Monte Carlo simulation (labeled 'MC'). The results are
compared for different combinations of $\a$ and $\beta$ values.}
\label{CompMCDF}
\eefig
As a second verification,
we compare our numerical solutions to the MET problem \eqref{exit}
and \eqref{dec} with those obtained from solving the SDE \eqref{SDE}
directly using Monte Carlo method. Both MET solutions are shown
in Fig.~\ref{CompMCDF}
for $\beta =0.5$ or $1$ and $\a = 0.5, 1, 1.5$ with $f\equiv 0, d=0$ and
the domain $D = (-1, 1)$.
To obtian the Monte Carlo solutions, we have taken
$10,000$ sample paths and the time step size $\Delta t = 0.001$ for each
starting point $x$ within $D$. The results in Fig.~\ref{CompMCDF} show
that the numerical solutions obtained from two different methods agree well.
This numerical experiment demonstrates that our numerical method
for computing the macroscopic quantities such as the MET is much more
computational efficient than the Monte Carlo method, because our numerical
error decreases quickly as the number of subintervals $J$ increases but
the Monte Carlo solution converges slowly
as the number of sample paths increases as expected.
For comparing the computational costs of the two methods,
we show the CPU time in Table~\ref{RunErr} for the Monte Carlo method
for the case we know the analytic solution for the MET:
 $\a=1.5, \beta=0, f\equiv 0, d=0,\e=1$.
 Table~\ref{RunErr} shows that, in order to achieve the two-significant-digit
 accuracy in MET, one needs more than $8000$ sample paths and
 the CPU time for the Monte Carlo simulations is about $890$ seconds
 for one starting point $x=0$ while our method needs only
 $1.3$ seconds for $40$ starting positions
 and the error is less than $1.4\times10^{-3}$.

\begin{table}[!hbp]




\caption{\label{RunErr} The CPU times and the errors in the Monte Carlo simulation with $\a=1.5, \beta = 0, f\equiv 0, d=0, \e=1, x=0$.}
\centering
\begin{tabular}{c|c|c|c|c}

\hline

\hline
   sample paths  &  1000    &   2000 & 4000 &8000          \\
\hline
time(s)   &  95.9    &    163 & 320.6 & 891.8   \\
\hline
Error($10^{-2}$) &  4.12   &  1.64 & 1.03 & 0.81     \\
\hline

\end{tabular}
\end{table}

\begin{figure}[h]
  \includegraphics*[width = \linewidth]{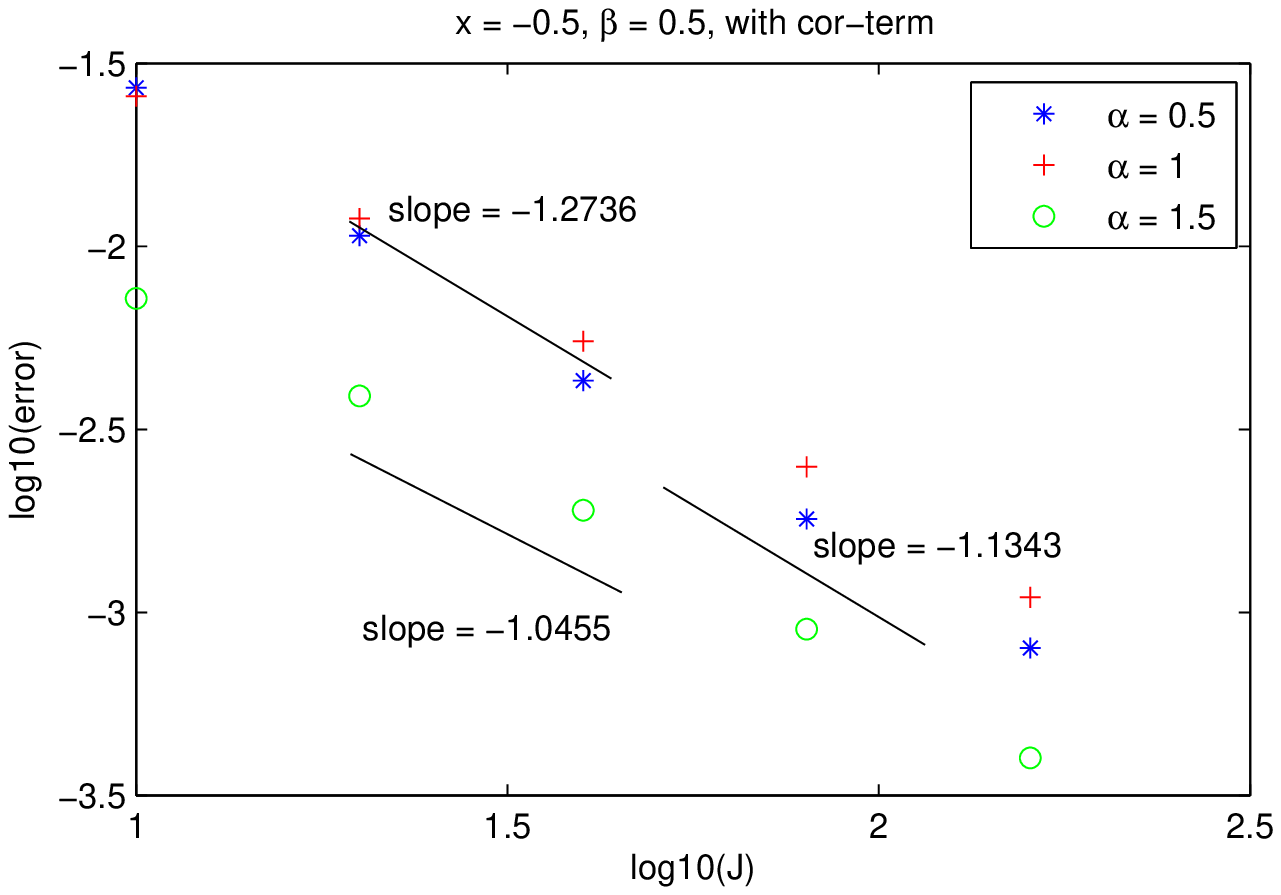}
\caption{Errors for the numerical solutions to the MET
for different resolutions $J=10, 20, 40, 80$ and
$160$ and $\a=0.5, 1, 1.5$ with $\beta = 0.5$, $f\equiv 0$, $d=0$ and $D=(-1,1)$. }
\label{asymm_order}       
\end{figure}
To estimate the convergence orders for computing the MET,
we take the numerical solution for the high resolution $J = 1280$
as the ''true'' solution. 
Figure~\ref{asymm_order} shows the errors of the solutions at $x = -0.5$
as $J$ increases from $10$ to $160$ by doubling,
for different values of $\a = 0.5, 1, 1.5$,
$\beta = 0.5$, $f\equiv 0$, $d=0$ and $D=(-1,1)$.
According to the results, we find that our numerical method can only reach
first-order accuracy and the numerical error for $\a = 1.5$ is smaller than
those in the cases of $\a = 0.5, 1$. Unlike the previous solution $(1-x^2)_+$,
the MET solution is known to have divergent derivatives at the boundary
points. \cite{Ting12}
Consequently, the error analysis in \eqref{eq.tee} does not apply
as the derivatives of the solution are unbounded.

\subsection{Mean exit time}

    For processes that are affected by asymmetric L\'evy noise,
little is known about the effects
of the different factors in the L\'evy noise on the MET. In this section,
we  examine the effects of the index of stability $\a$,
 the skewness parameter $\beta$, the size of the domain $(-b,b)$,
the drift $f$, the intensity of Gaussian part $d$ and non-Gauassian part
$\eps$ of the noise.
For odd drift functions $f(x)$, as shown in Proposition~\ref{thm}, the solutions
are symmetric when $\beta$ changes sign,
i.e.
$u_{-\beta}(-x) = u_\beta(x)$. Thus, we will only present the results
for $\beta \geq 0$.

\subsubsection{Effects of the skewness parameter $\beta$ and the index of stability $\a$}


\befig
\includegraphics*[ width=\textwidth]{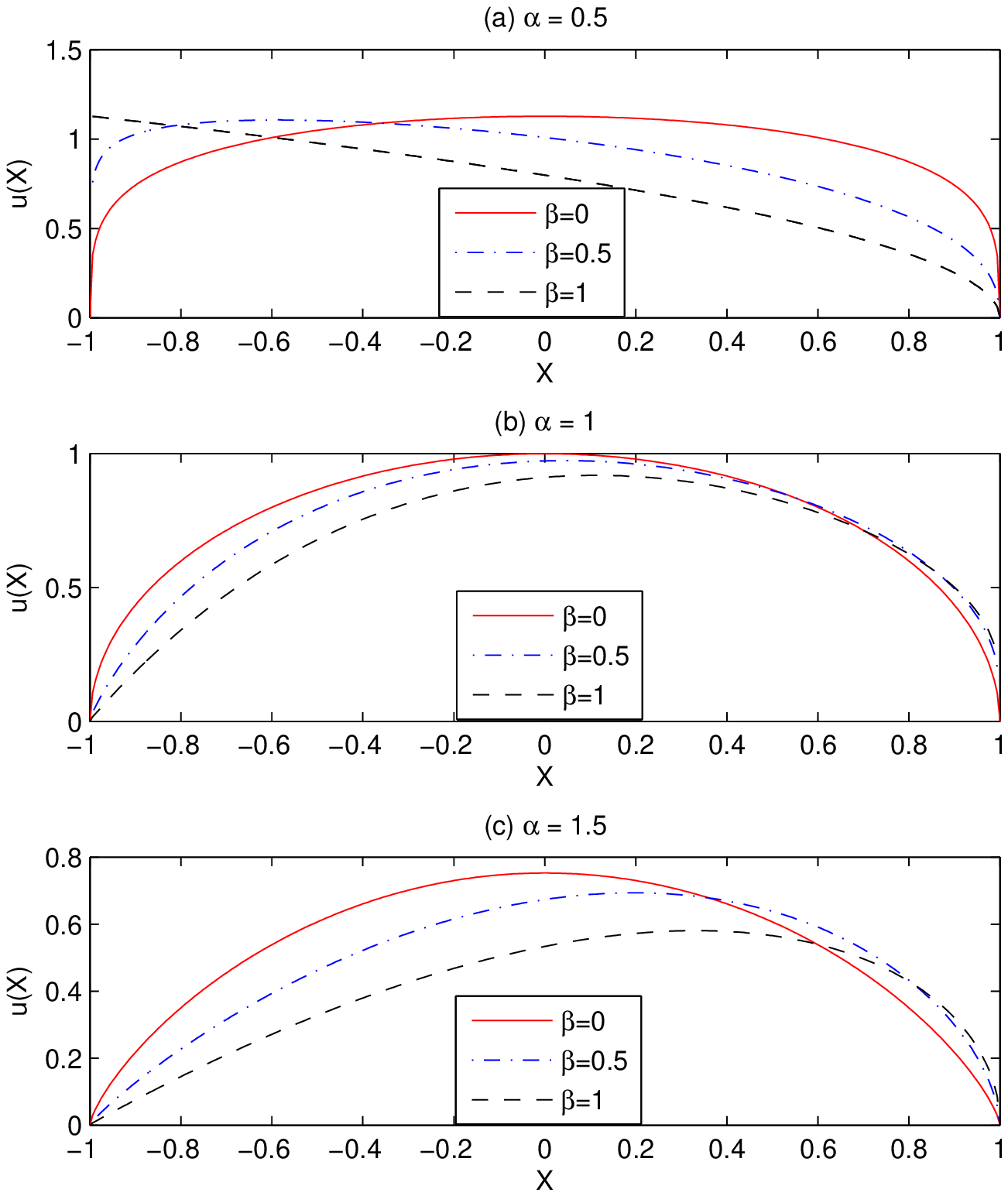}
\caption{The effect of the skewness parameter $\beta$
on the MET $u(x)$ with pure jump L\'evy measure($d = 0$,
$\eps = 1$, and $f \equiv 0$),  domain $D = (-1, 1)$,
for different $\a$ ($\a = 0.5$(part(a)), $\a = 1$(part(b)), $\a = 1.5$(part(c)))  and $\beta$ ($\beta = 0$ (solid line),
$\beta = 0.5$ (dashdot line), $\beta = 1$ (dashed line)).
}
\label{change_betaf0d0}
\eefig
Figure~\ref{change_betaf0d0} shows that the MET solutions
for different values of $\beta=0, 0.5, 1$ for each of $\a=0.5, 1$ and $1.5$.
In all cases shown, the domain is $D=(-1,1)$ and there is no drift $f\equiv 0$
and no Gaussian part $d=0$.
When $\beta = 0$,
the MET $u$ is known, given by $\displaystyle{u(x) = \frac{\sqrt{\pi}}{2^\a \Gamma(1+\a/2) \Gamma((1+\a)/2)} (1-x^2)^{\a/2}}$,
symmetric about $x=0$ just like the PDF
$S_\a(1,0,0)$ displayed in Fig.~\ref{skewpdf}.
When $\beta \neq 0$, the MET is not symmetric about $x=0$
even when the domain is symmetric. Furthermore, the larger $\beta$ is, the more asymmetric MET is.
There are significant differences on the MET due to the effect of $\beta$
for different values of $\alpha$. For $\a=0.5$ and $\beta=0.5$ or $1$,
we find that the METs $u(x)$
are discontinuous at the left boundary $x=-1$, implying that the MET is nonzero
once the starting point is insider the domain. The MET is smaller for larger value of $\beta$ if the starting point $x$ is positive, while the MET is much larger for bigger $\beta$ when the starting point is close to the left boundary.
In contrast, the behavior changes for $\alpha =1$ or $1.5$, as shown in Fig.~\ref{change_betaf0d0}(b) and (c): the MET is mostly smaller for larger value
of $\beta$ for most of the starting points except when the starting point is
close to the right boundary. These behaviors can be explained by examining
the corresponding PDFs shown in Fig.~\ref{skewpdf}.
For example, the PDFs $S_{0.5}(1,\beta,0)$ in Fig.~\ref{skewpdf}(a)
show that it has almost zero
probability moving to the left for $\beta=1$, while the PDFs of
$S_{1.5}(1,\beta,0)$ in Fig.~\ref{skewpdf}(c)
indicate that the stochastic process has
larger probability moving to its immeadiate left when $\beta$ is larger.

\befig[h]
\includegraphics*[width = \linewidth]{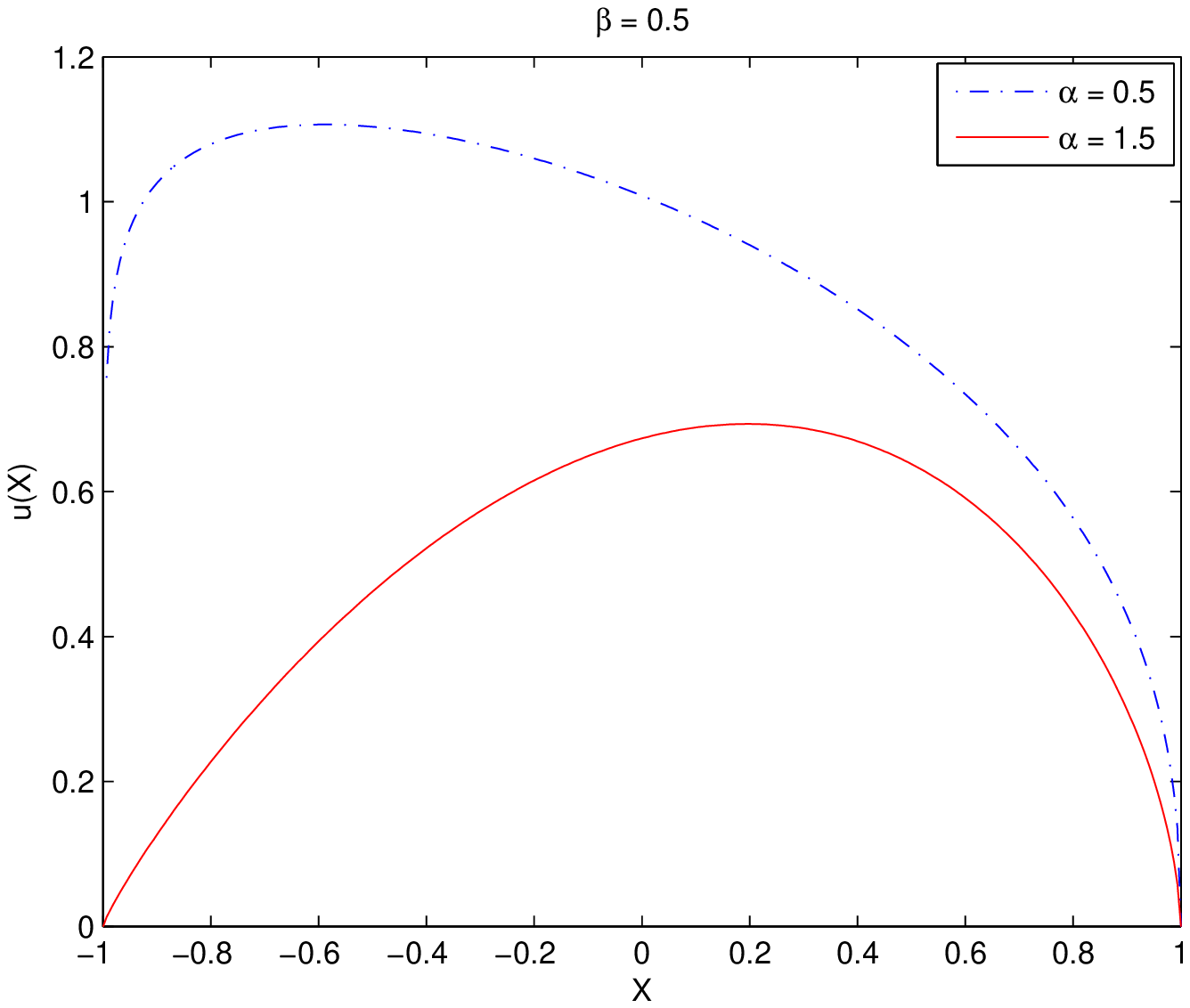}
\caption{The MET $u(x)$ with $d = 0$, $f \equiv 0$, $\eps = 1$ and domain $D = (-1, 1)$
for $\beta = 0.5$ and different values of $\a = 0.5, 1.5$. }
\label{diffalpBeta05}
\eefig
To show the effect of index of stability $\a$ directly,
Fig.~\ref{diffalpBeta05}
plots the METs for the fixed skew parameter $\beta = 0.5$
but two different values of $\alpha=0.5$ and $1.5$ in one graph.
For all starting point $x$, the MET is smaller, skewed toward to the right
and a continuous function of $x$ when $\a=1.5$, while
that of $\a=0.5$ is skewed toward the left and discontinuous
at the left boundary. Compared with the PDFs of $S_\a(1,0.5,0)$
in Fig.\ref{skewpdf}(b), the process for $\a=0.5$ has much smaller
chance to move to the left, thus it takes longer time to
exit the domain for the initial starting position
in the left part of the domain.

\subsubsection{Effect of domain size}

\befig
\includegraphics*[width = \linewidth]{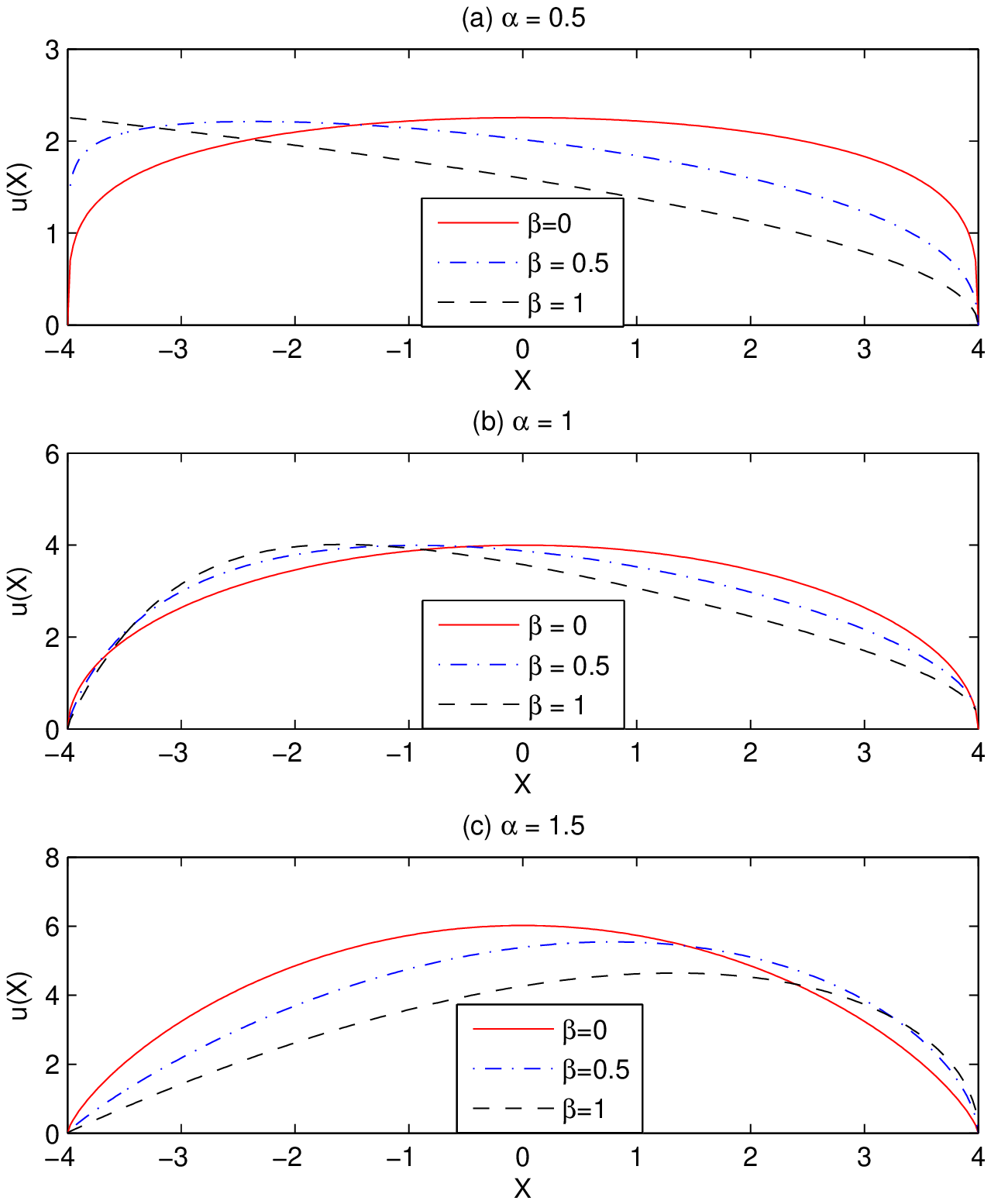}
\caption{The MET $u(x)$ for the larger domain $D = (-4, 4)$ with
pure jump L\'evy measure($d = 0$, $\eps = 1$ and $f\equiv 0$)
for different values of $\beta$. (a) $\a=0.5$; (b) $\a=1$; (c) $\a=1.5$.}
\label{change_betaf0d0b4}
\eefig
Next, we increase the domain size to $D = (-4, 4)$
and keep the other factors the same as in Fig.~\ref{change_betaf0d0},
i.e. $d = 0, \eps = 1$ and $f \equiv 0$ for $\a = 0.5, 1, 1.5$.
Comparing the results corresponding to the different domain sizes
(the smaller size in Fig.~\ref{change_betaf0d0} and
the larger size in Fig.~\ref{change_betaf0d0b4}),
we find that behaviors of the MET for different values of $\beta$ are
similar for $\a=0.5$ and $\a=1.5$. However, the profiles of the METs
for $\a=1$ are dramatically different when the domain $D$ changes
from $(-1,1)$ to $(-4,4)$.  For the larger domain, the process starting
from the most of the left-half of the domain takes longer time
to exit the domain when $\beta$ increases, while the opposite is true
for the smaller domain. For the same value of $\beta$, the shapes of the MET
skewed toward the left for $D=(-4,4)$ instead of toward to the right
for $D=(-1,1)$.

\subsubsection{Effect of noises }
\befig[h]
\includegraphics*[width = \linewidth]{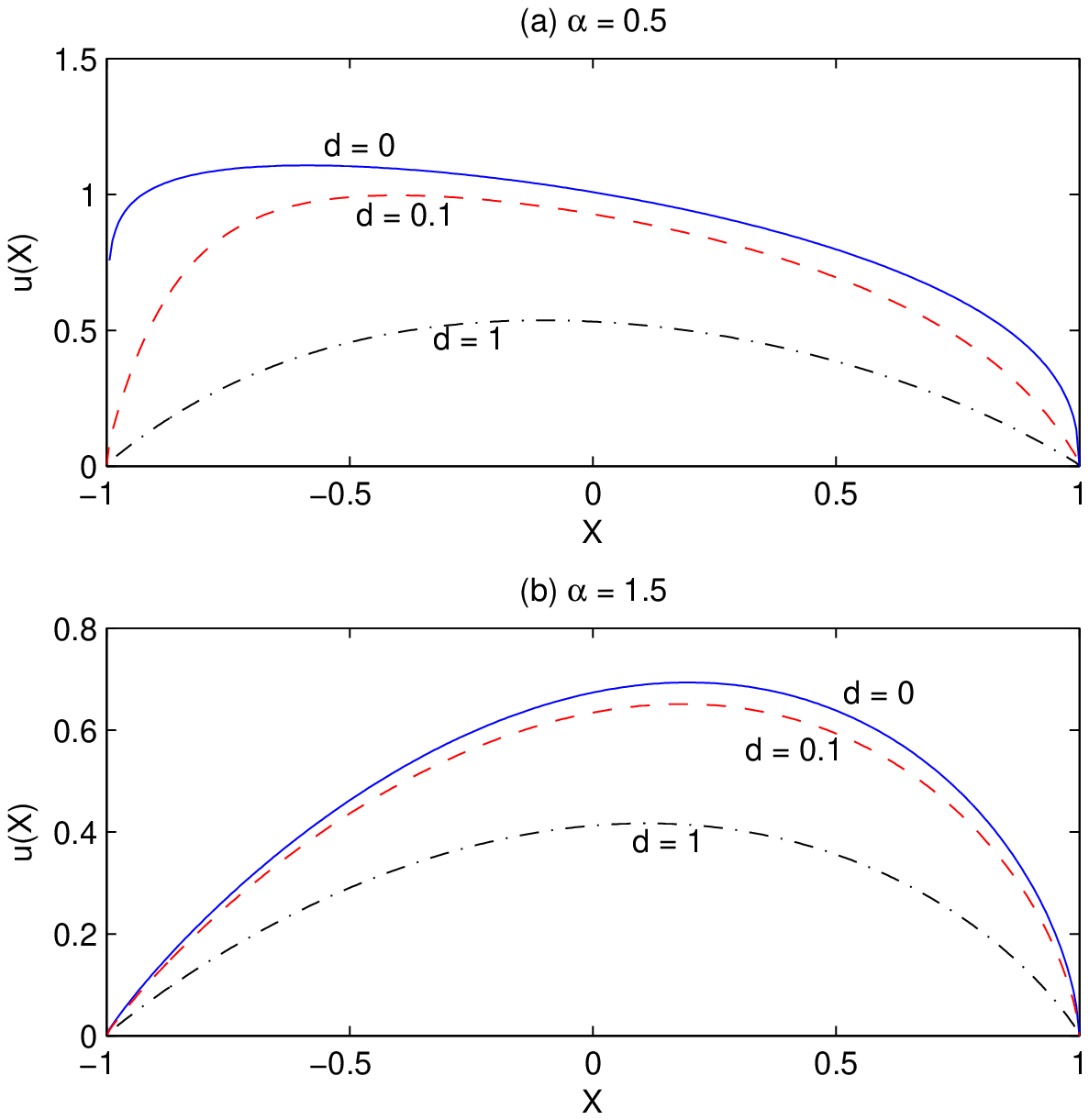}
\caption{The effect of Gaussian noise. The METs are plotted
for $d = 0$(blue solid line), $d = 0.1$(red dashed line), $d = 1$(green dashdot line) with the domain $D = (-1,1)$, the skewness parameter $\beta = 0.5$, the drift term $f \equiv 0$ and $\eps = 1$ for $\a = 0.5$(part(a)) and $\a = 1.5$(part(b)). }
\label{change_d_f0r1}
\eefig
Now, let's consider the effect of the intensities of
the Gaussian noise, $d$, and the non-Gaussian noise, $\eps$.
Figure~\ref{change_d_f0r1} shows the METs for different values
of $d = 0, 0.1, 1$ with the domain $D = (-1, 1)$,
 $\eps = 1, f \equiv 0$ and  $\beta = 0.5$.
The role of the Gaussian noises play on MET is obvious:
when the noise is stronger, the MET is shorter for any $\a$,
similar to the results for the symmetric L\'evy cases shown in \cite{Ting12}.
If we add any amount of Gaussian noise (even for the low intensity $d = 0.1$),
the MET becomes continuous at the left end point when $\a = 0.5$.
As the intensity of Gaussian increases, the MET is more symmetrical about
the center of the domain $x=0$.

\befig[h]
\includegraphics*[width = \linewidth]{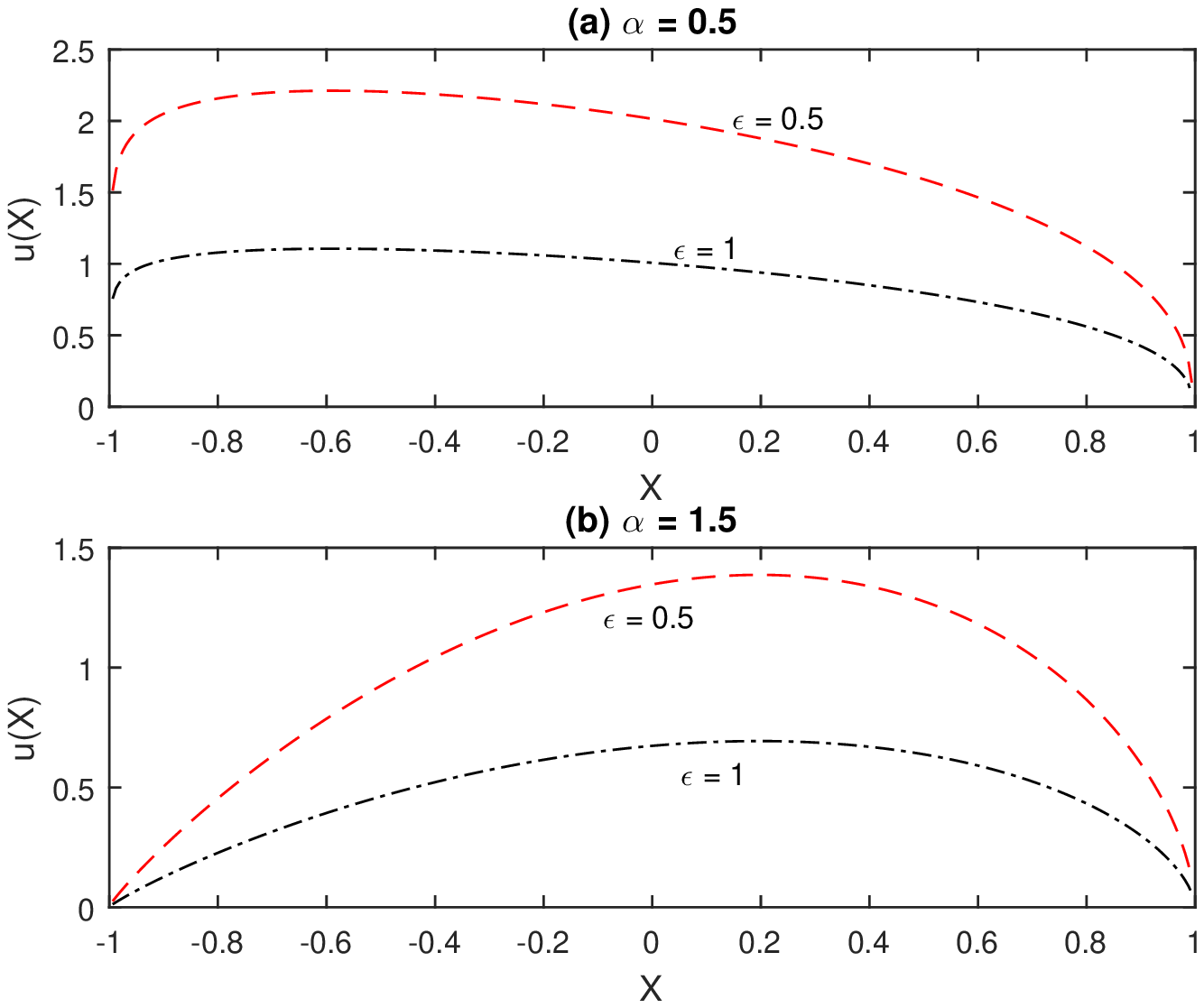}
\caption{The effect of the non-Gaussian noise
on the MET with pure asymmetric L\'evy motion($f\equiv 0, d = 0$).
(a) The METs with the domain $D = (-1,1)$, $\a = 0.5$ and the skewness parameter $\beta = 0.5$
 for different values of $\eps = 0.5$ (the dashed line) and $\eps = 1$ (the dash-dotted line); (b) The same as (a) except $\a = 1.5 $. }
\label{change_eps}
\eefig
Figure~\ref{change_eps} shows the effect of the intensity
of the non-Gaussian noise. From the METs for $\eps = 0.5, 1$
with domain $D = (-1,1)$, $\beta = 0.5$,
$f \equiv 0$ and $\a = 0.5$ and $\a = 1.5$, the MET gets smaller
when $\eps$ increases and the shape profile of the MET does not change
much as we $\eps$ changes.

\subsubsection{Effect of drift term $f$}

Last, we examine the effect of the O-U potential, i.e., having  the drift term
$f(x)=-x$ on the MET. Figure~\ref{ou_d0beta05} shows the METs with the drift
$f(x)=-x$ and without the drift $f\equiv 0$ for the case of
pure non-Gaussian noise ($d=0$ and $\eps=1$). In the presence of the O-U potential, the MET increases as expected.
For $\a=0.5$, the MET $u(x)$ becomes discontinuous at both the boundary points of the domain, $x=1,-1$  when the O-U potential is added to the system. In contrast, the MET stays continuous for $\a=1.5$ when the drift term is present.
Again, our numerical results demonstrate that
the regularity of the solution appears to be dependent on whether
$\a$ is greater than 1 or less than 1. It would be interesting research topic
to investigate it theorectically.

\befig[h]
\includegraphics*[width = \linewidth]{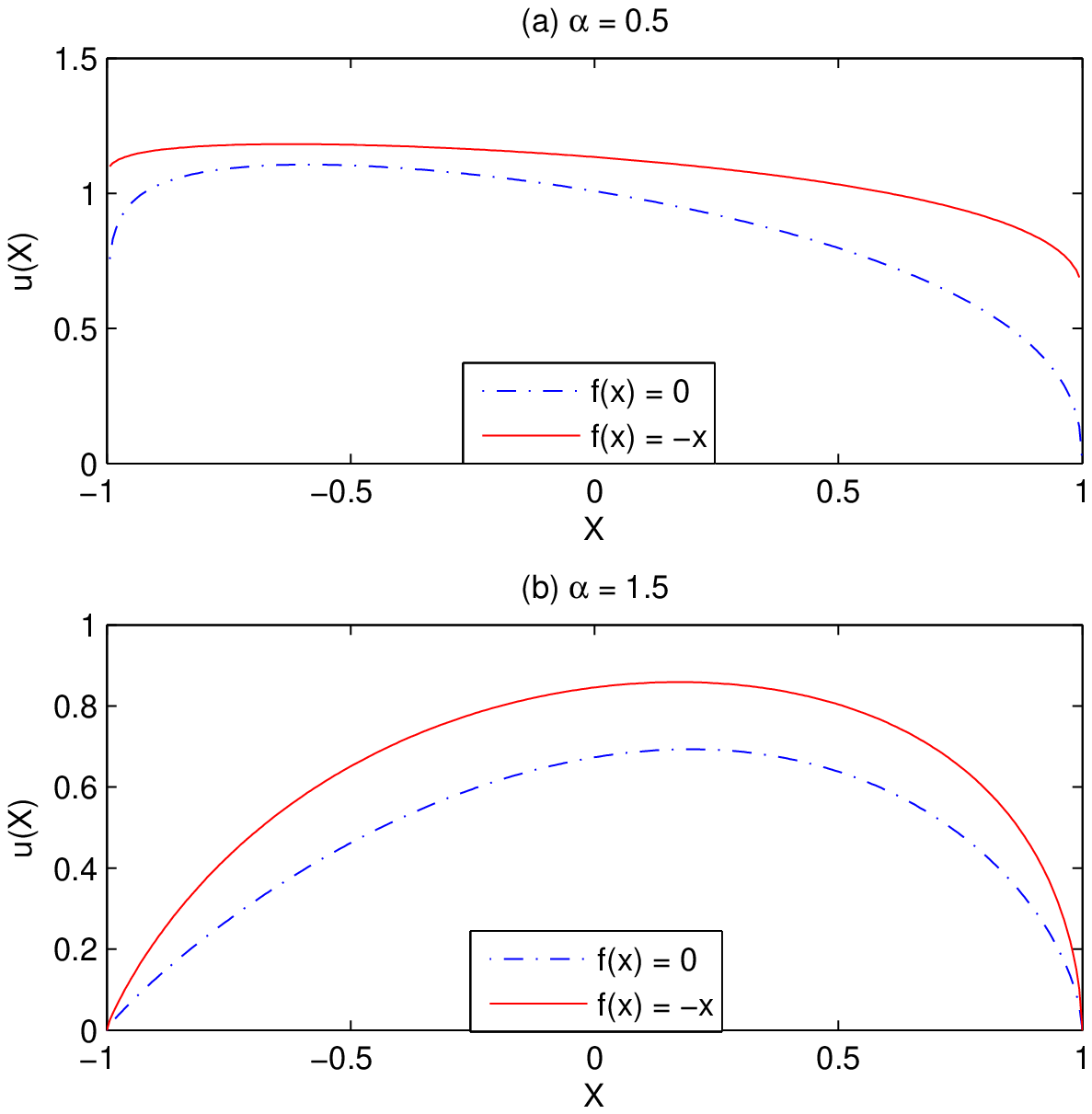}
\caption{The effect of the drift term on the MET with the domain $D = (-1,1)$
$d = 0$, $\eps = 1$ and the skewness parameter $\beta = 0.5$.
(a) The METs for $\a = 0.5$ with the drift $f(x)=-x$ (the solid line) and without the drift $f\equiv 0$ (the dash-dotted line).
(b) The same as (a) except $\a = 1.5$. }
\label{ou_d0beta05}
\eefig

\subsection{Escape probability}

\befig
\includegraphics*[width = \linewidth]{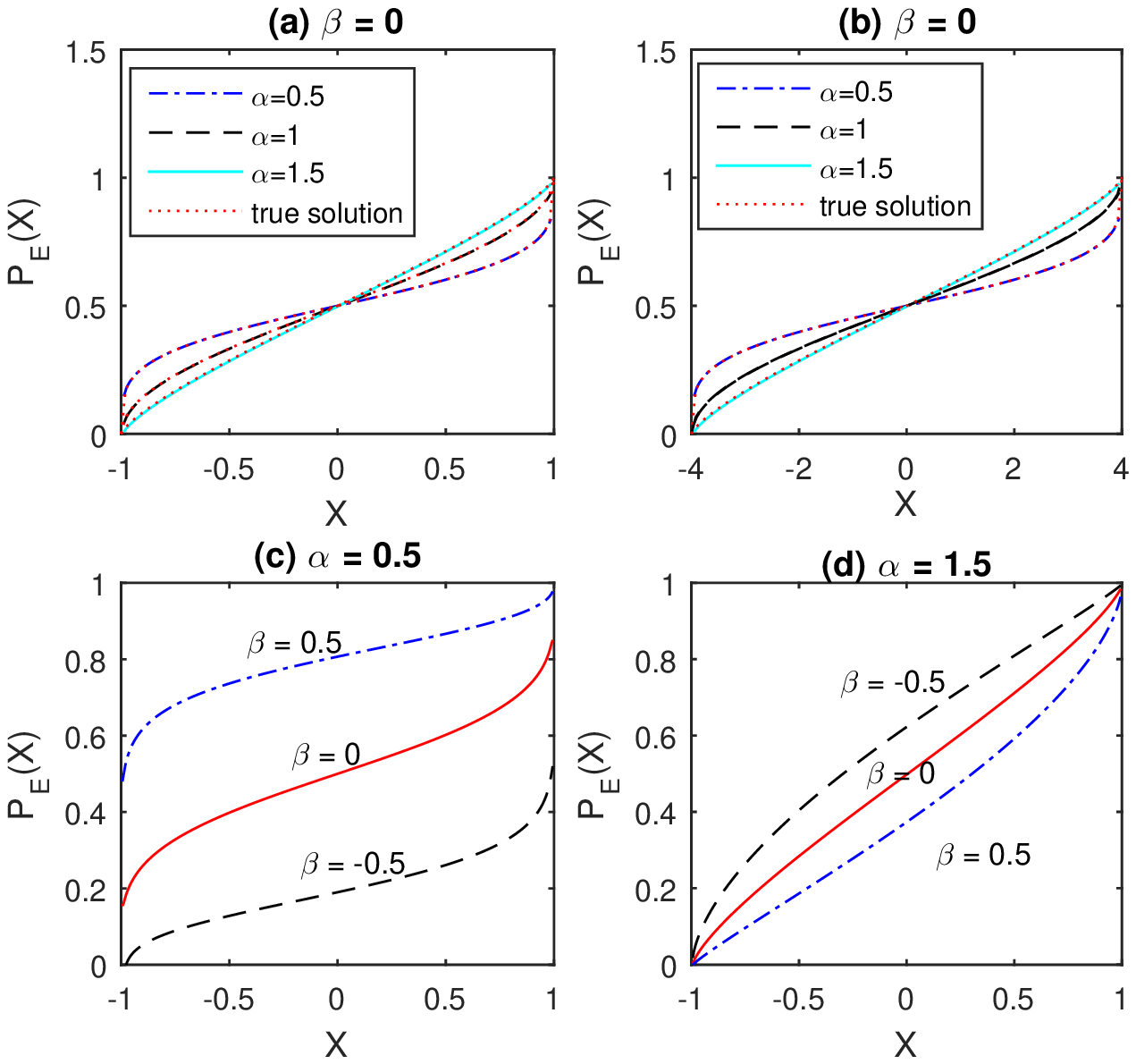}
\caption{The escape probability from $D$ first landing in $E$
for $d =0$, $f \equiv  0$ and $\eps = 1$.
(a) Symmetric L\'evy motion($\beta = 0$) and domain $D = (-1,1)$,
    $E = (1,\infty)$ for different $\a$.
(b) The same as (a) except $D = (-4,4)$ and $E = (4,\infty)$.
(c) $\a = 0.5$ and domain $D = (-1,1)$, $E = (1, \infty)$
for the different values of $\beta=-0.5, 0, 0.5$.
(d) The same as (c) except  $\a = 1.5$. }
\label{EpDifBeta}
\eefig
First, we validate our equations \eqref{LastEq11} and \eqref{LastEq12}
and the numerical implementation by comparing
 with analytical result for the special case
of $f \equiv 0$, $d = 0$, $\eps =1$ and $\beta=0$ \cite{Ting12}
\bear
    P_E(x) = \frac{(2b)^{1-\a}\Gamma(\a)}{[\Gamma(\frac{\a}{2})]^2}
    \int_{-b}^x(b^2-y^2)^{\frac{\a}2-1} \dy.
\enar
Figure~\ref{EpDifBeta}(a) and (b) show our numerical results match the
analytical solution well for different domains, which verifies that
our equations and the computer programs are correct.

Figure~\ref{EpDifBeta}(c) and (d)
tell us the effect of the skewness parameter $\beta$ on the escape probability
of a process starting from the domain $D=(-1,1)$, exiting the domain and
landing first to the right $E=(1,\infty)$. Here, we consider
the pure jump L\'evy process: $d=0, f\equiv 0$ and $\eps=1$.
The results show that, when $\a=0.5$, the escape probability
is a sensitive function of $\beta$ and increases dramatically as
$\beta$ changes from $-0.5$ to $0$ then to $0.5$.
In the contrast, when $\a=1.5$,  the process has less chance to escape the
domain $D$ and first land in $E$. The results are consistent with
the plots of the probability density functions $S_{0.5}(1,\beta,0)$
and $S_{1.5}(1,\beta,0)$ shown in Fig.~\ref{skewpdf}(a,c) and those of
METs in Fig.~\ref{change_betaf0d0}. Here we note that, when $\a=0.5$,
the escape probability seems to be discontinuous at the left boundary
point of the domain for $\beta=0.5$ and at the right boundary point
for $\beta=-0.5$.
However, the probability is a continuous function of the initial position $x$ when $\a=1.5$, agreeing with the results of MET presented earlier in the section.

\befig[h]
\includegraphics*[width = \linewidth]{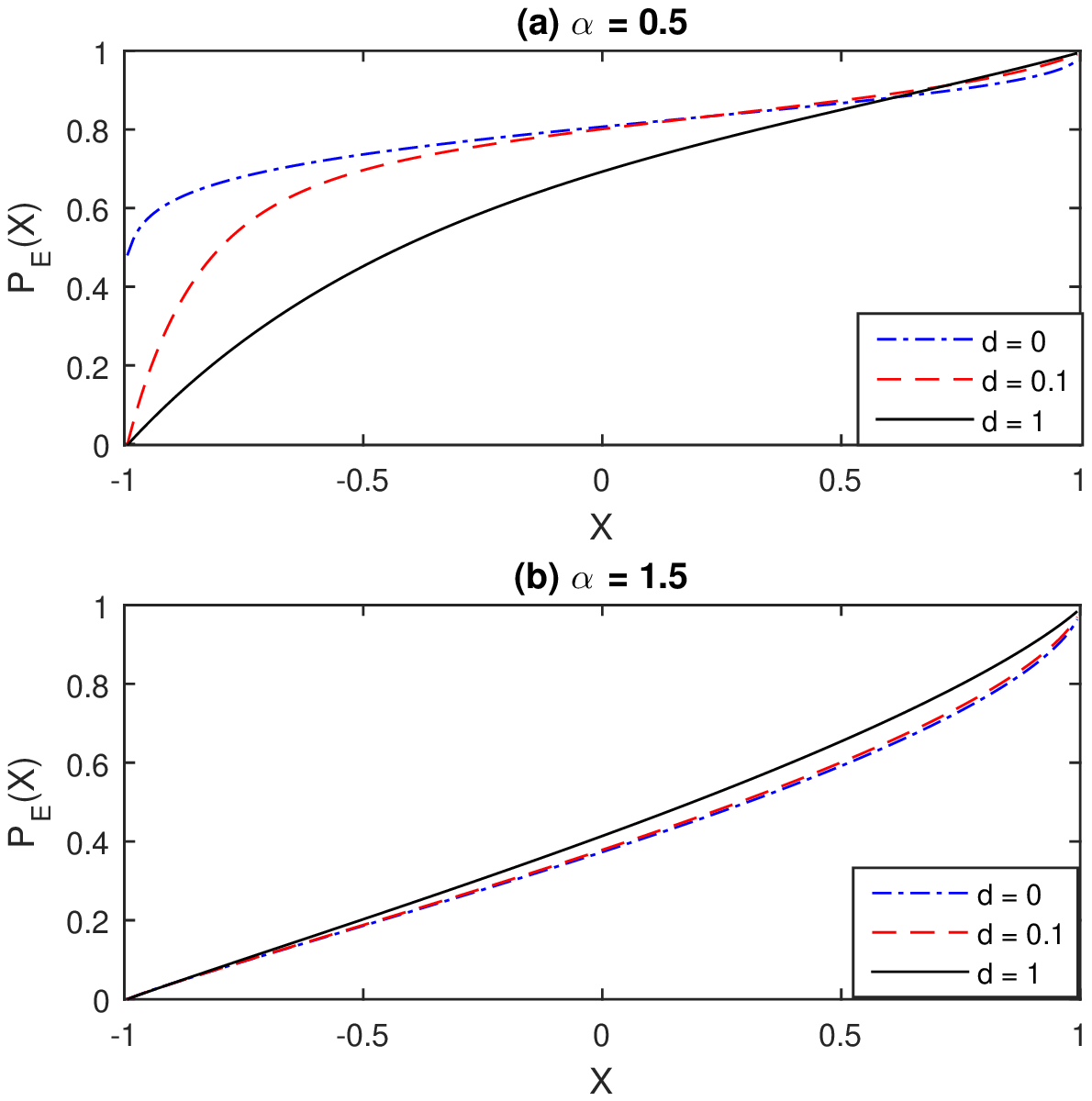}
\caption{The effect of the Gaussian noise ($d = 0, 0.1, 1$) on the probability of escaping from $D = (-1,1)$ first landing in $E = (1,\infty)$.
(a) $\a = 0.5$, $\beta = 0.5$, $\eps = 1, f \equiv 0$.
(b) Same as (a) except $\a = 1.5$.}
\label{EpDiff_drift}
\eefig

\befig[h]
\includegraphics*[width = \linewidth]{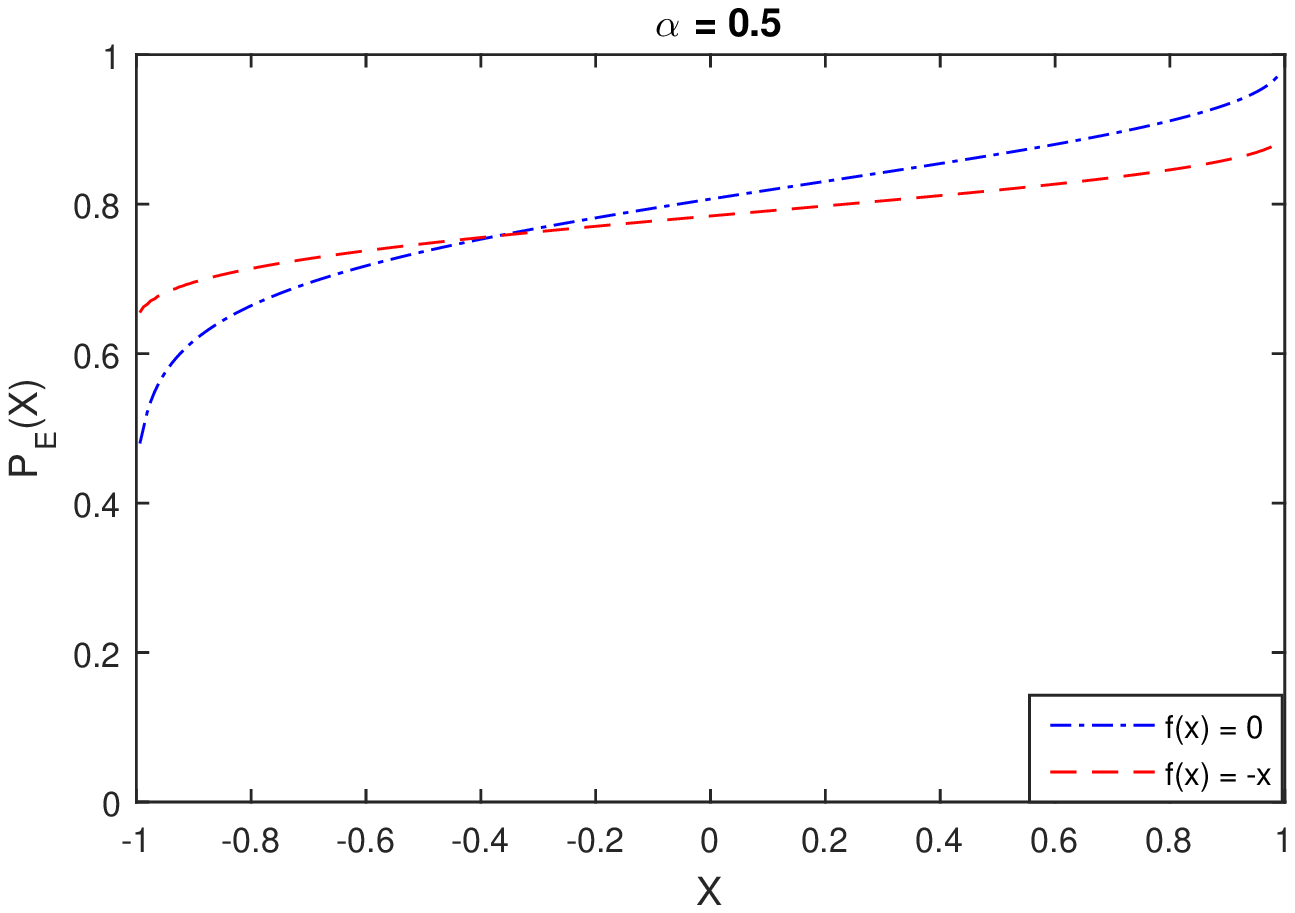}
\caption{The effect of the drift term $f$ on the escape probability
from $D = (-1,1)$ first landing in $E = (1,\infty)$
for $d = 0$,  $\eps=1$, $\beta = 0.5$, $\a=0.5$ with
for two different $f\equiv 0$ and $f(x)=-x$.
}
\label{EpChange_drif}
\eefig

Next, we consider the effect of the intensity of the Gaussian noise and the drift term on the escape probability
in the case of asymmetric L\'evy noise with $\beta = 0.5$.
For $\a=0.5$, Fig.~\ref{EpDiff_drift}(a) shows
that, with any amount of the Gaussian noise,  the escape probability becomes
a smoother function eliminating any discontinuities at the boundary
and it has larger effect for negative starting position $x$ than that
for positive $x$. The impact of the Gaussian noise is relatively
small for $\a=1.5$ as shown in Fig.~\ref{EpDiff_drift}(b).
Figure~\ref{EpChange_drif} shows that
the difference in the drift term $f$ has a great influence
on the escape probability.
Recall that the drift term $f(x)=-x$ drives the process toward the globally
stable point $x=0$. Comparing with zero drift $f\equiv 0$,
we find that the probability with the drift term is smaller
for the starting position $x>-0.3$, while the escape probability
with the drift is larger than that without the drift for the starting position
$x<-0.5$.
Similar to the results of MET, the presence of the O-U potential $f(x)=-x$
causes that the escape probability be discontinuous at both end points
of the domain.


\section{Conclusions}
  For asymmetric L\'evy motions, there are important applications in many fields, such as
physics, mathematical finance and insurance risks. In addition, it attracts the attention
of mathematicians because it is closely related to the nonlocal partial differential equations.
In this work, an effective and convergent numerical algorithm has been developed for solving the mean
first exit time and escape probability for one-dimensional stochastic systems with asymmetric L\'evy motion.
The convergence of the numerical method is verified numerically and the pointwise convergence order is closed to first-order.
The numerical analysis predicts that the method is of second-order accuracy
for smooth solutions. However, the numerical solutions have divergent derivatives or are discountinous at the boundary and therefore the convergence order suffers.

 We also consider the influence of different parameters of the system on the mean exit time and the escape probability.
For certain deterministic drift and symmetric domains,
we find that the MET has a symmetry with respect to the skew parameter $\beta$.
Thus we focus on the case of $\beta > 0$ and have seen that the profile of the
MET becomes more asymmetric as $\beta$ is larger. From our numerical results,
we find that, for $0<\a<1$, the MET and the escape probability
appear to be discontinous at the boundary of the domain but they are
continuous for $\a>1$.
This interesting behavior of the solution worths further theoretical
investigation.

\section{Acknowledgements}
The research is partially supported by the grants China Scholarship Council (X.W.), NSF-DMS \#1620449 (J.D. and X.L.),
and Simons Foundation \#429343 (R.S.).

\bibliographystyle{plain}
\bibliography{meanexit}

\begin{thebibliography}{10}

\bibitem{Apple}
D.~Applebaum.
\newblock {\em L{\'e}vy Processes and Stochastic Calculus}.
\newblock Cambridge University Press, Cambridge, U.K., 2nd edition, 2009.

\bibitem{completely_asy1}
J.~Bertoin.
\newblock On the first exit time of a completely asymmetric stable process from
  a finite interval.
\newblock {\em B. Lond. Math. Soc}, 28(5):514--520, 1995.

\bibitem{exponential_decay_asy}
J.~Bertoin.
\newblock Exponential decay and ergodicity of completely asymmetric {L}{\'e}vy
  processes in a finite interval.
\newblock {\em Ann. Appl. Probab}, 7(1):156--169, 1997.

\bibitem{Chen15}
Z.~Chen, R.~Song, and X.~Zhang.
\newblock Stochastic flows for {L}{\'e}vy process with {H}\"older drifts.
\newblock {\em arXiv:1501.04758}, 2015.

\bibitem{Du2012}
Q.~Du, M.~Gunzburger, R.~B. Lehoucq, and K.~Zhou.
\newblock Analysis and approximation of nonlocal diffusion problems with volume
  constraints.
\newblock {\em SIAM Review}, 54(4):667--696, 2012.

\bibitem{Duan}
J.~Duan.
\newblock {\em An Introduction to Stochastic Dynamics}.
\newblock Cambridge University Press, 2014.

\bibitem{finite}
B.~Dybiec, E.~Gudowska-Nowak, and P.~H\"anggi.
\newblock L{\'e}vy-{B}rowmian motion on finite intervals: Mean first passage
  time analysis.
\newblock {\em Phys. Rev. E}, 73:046104, 2006.

\bibitem{Ting12}
T.~Gao, J.~Duan, X.~Li, and R.~Song.
\newblock Mean exit time and escape probability for dynamical systems driven by
  {L}\'evy noise.
\newblock {\em SIAM J. Sci. Comput.}, 36(3):A887--A906, 2014.

\bibitem{Mengli}
M.~Hao, J.~Duan, R.~Song, and W.~Xu.
\newblock Asymmetric non-{G}aussian effects in a tumor growth model with
  immunization.
\newblock {\em Appl. Math. Model.}, 38:4428--4444, 2014.

\bibitem{asymmetric_Hein}
C.~Hein, P.~Imkeller, and I.~Pavlyukevich.
\newblock Limit theorems for p-variations of solutions of sdes driven by
  additive stable {L}{\'e}vy noise and model selection for paleo-climatic data.
\newblock {\em Interdiscip. Math. Sci.}, 8:137--150, 2009.

\bibitem{Huang}
Y.~Huang and A.~Oberman.
\newblock Numerical methods for the fractional {L}aplacian: A finite
  difference-quadrature approach.
\newblock {\em SIAM J. Numer. Anal.}, 52(6):3056--3084, 2014.

\bibitem{Klup04}
C.~Kluppelberg.
\newblock A continuous-time {GARCH} process driven by a {L}{\'e}vy process:
  stationarity and second-order behavior.
\newblock {\em J. Appl. Prob.}, 41(3):601--622, 2004.

\bibitem{first}
T.~Koren, A.~Chechkin, and J.~Klafter.
\newblock On the first passage time and leapover properties of {L}{\'e}vy
  mmotion.
\newblock {\em Physica A}, 379:10--22, 2007.

\bibitem{Lambert}
A.~Lambert.
\newblock Completely asymmetric {L}{\'e}vy processes confined in a finite
  interval.
\newblock {\em Ann. I. H. Poincare-Pr.}, 36(2):251 -- 274, 2000.

\bibitem{SpecAppFracDer}
C.~Li, F.~Zeng, and F.~Liu.
\newblock Spectral approximations to the fractional integral and derivative.
\newblock {\em Fract. Calc. Appl. Anal.}, 15(3):383--406, 2012.

\bibitem{Lu14}
X.~L\"{u} and W.~Dai.
\newblock Stochastic partial differential equations driven by fractional
  {L}{\'e}vy noises.
\newblock {\em arXiv:1410.09992v1}, 2014.

\bibitem{temper1}
Y.~S.~Kim M.~L.~Bianchi, S. T.~Rachev and F.~J. Fabozzi.
\newblock {\em Tempered stable distributions and processes in finance:
  numerical analysis}.
\newblock Springer Milan, 2010.

\bibitem{FPDE_Shen}
Z.~Mao and J.~Shen.
\newblock Efficient spectral-{G}alerkin method for fractional partial
  differential equations with variable coefficients.
\newblock {\em J. Comput. Phys}, 307, 2016.

\bibitem{Middleton}
D.~Middleton.
\newblock Non-{G}aussian noise models in signal processing for
  telecommunications: new methods an results for class {A} and class {B} noise
  models.
\newblock {\em IEEE T. Inform. Theory}, 45(4):1129--1149, 1999.

\bibitem{Asy_financial_ratios}
B.~Podobnik, A.~Valentincic, D.~Horvatic, and H.~Stabley.
\newblock Asymmetric l{\'e}vy flight in financial ratios.
\newblock {\em Proc. Natl. Acad. Sci. U.S.A.}, 108(11):17883¨C17888, 2011.

\bibitem{Tankov}
J.~Poirot and P.~Tankov.
\newblock Monte {C}arlo option pricing for tempered stable ({CGMY}) processes.
\newblock {\em Asia-Pac. Financ. Markets}, 13:327--344, 2006.

\bibitem{Priola12}
E.~Priola.
\newblock Pathwise uniqueness for singular {SDE}s driven by stable processes.
\newblock {\em Osaka J. Math.}, 49(2):421--447, 2012.

\bibitem{Asymptotic_Qiao}
H.~Qiao and J.~Duan.
\newblock Asymptotic methods for stochastic dynamical systems with small
  non-{G}asusian {L}{\'e}vy noise.
\newblock {\em Stoch. Dynam.}, 15, 2015.

\bibitem{Huijie}
H.~Qiao, X.~Kan, and J.~Duan.
\newblock Escape probability for stochastic dynamical systems with jumps.
\newblock In F.~Viens, J.~Feng, Y.~Hu, and E.~Nualart, editors, {\em Malliavin
  Calculus and Stochastic Analysis}, volume~34, pages 195--216, 2013.

\bibitem{GMRES}
Y.~Saad and M.H. Schultz.
\newblock A generalized minimal residual algorithm for solving nonsymmetric
  linear systems.
\newblock {\em SIAM J. Sci. Stat. Comput.}, 7:856--869, 1986.

\bibitem{Taqqu}
G.~Samorodnitsky and M.~S. Taqqu.
\newblock {\em Stable Non-Gaussian Random Process}.
\newblock Chapman \& Hall/CRC, 1994.

\bibitem{Sato-99}
K.~I. Sato.
\newblock {\em L{\'e}vy Processes and Infinitely Divisible Distributions}.
\newblock Cambridge University Press, 1999.

\bibitem{Sidi}
A.~Sidi and M.~Israeli.
\newblock Quadrature methods for periodic singular and weakly singular fredholm
  integral equtaions.
\newblock {\em J. Sci. Comput.}, 3(2):201--231, 1998.

\bibitem{METXiao}
X.~Wang, J.~Duan, X.~Li, and Y.~Luan.
\newblock Numerical methods for the mean exit time and escape probability of
  two-dimensional stochastic dynamical systems with non-{G}aussian noises.
\newblock {\em Appl. Math. Comput.}, 258(0):282 -- 295, 2015.

\bibitem{Xu2013}
Y.~Xu, J.~Feng, J.~Li, and H.~Zhang.
\newblock L{\'e}vy noise induced switch in the gene transcriptional regulatory
  system.
\newblock {\em Chaos}, pages 1--11, 2013.

\bibitem{Yang2010}
Q.~Yang, F.~Liu, and I.~Turner.
\newblock Numerical methods for fractional partial differential equations with
  riesz space fractional derivatives.
\newblock {\em Appl. Math. Model.}, 34(1):200 -- 218, 2010.

\end{thebibliography}

\end{document}